\documentclass[12pt]{amsart}
\usepackage{amsmath,bm}
\usepackage{amsfonts, amsmath, amssymb, amsgen, amsthm, amscd}
\usepackage{mathrsfs}

\usepackage{xcolor}
\usepackage{graphicx}
\newcommand{\Bmu}{\mbox{$\raisebox{-0.59ex}
  {$l$}\hspace{-0.18em}\mu\hspace{-0.88em}\raisebox{-0.98ex}{\scalebox{2}
  {$\color{white}.$}}\hspace{-0.416em}\raisebox{+0.88ex}
  {$\color{white}.$}\hspace{0.46em}$}{}}

\usepackage[all]{xy}

\usepackage[top=30mm, left=27mm, bottom=27mm, right=27mm]{geometry}

\usepackage{enumerate}

\def\d{\delta}
\def\D{\Delta}

\def\s{\sigma}

\def\ve{\varepsilon}

\def\Id{\mathop{\rm Id}\nolimits}

\def\lra{\longrightarrow}

\def\ot{\otimes}
\def\odots{\ot\cdots\ot}

\def\nb{\nabla}
\def\lra{\longrightarrow}

\def\rt{\triangleright}
\def\lt{\triangleleft}

\def\D{\Delta}

\def\Id{\mathop{\rm Id}\nolimits}

\newcommand{\ps}[1]{~\hspace{-4pt}_{^{(#1)}}}

\newcommand{\ns}[1]{~\hspace{-4pt}_{_{{<#1>}}}}

\newcommand{\nsb}[1]{~\hspace{-4pt}_{^{[#1]}}}

\usepackage{enumerate}

\newcommand{\C}[1]{\mathcal{#1}}

\newcommand{\wbar}[1]{\overline{#1}}

\newcommand{\ie}{{\it i.e.\/}\ }

\renewcommand{\leq}{\leqslant}
\renewcommand{\geq}{\geqslant}

\numberwithin{equation}{section}

\newtheorem{theorem}{Theorem}[section]
\newtheorem{proposition}[theorem]{Proposition}
\newtheorem{lemma}[theorem]{Lemma}
\newtheorem{corollary}[theorem]{Corollary}
\theoremstyle{definition}

\newtheorem{example}[theorem]{Example}

\DeclareFontFamily{OT1}{pzc}{}
\DeclareFontShape{OT1}{pzc}{m}{it}{<-> s * [1.200] pzcmi7t}{}
\DeclareMathAlphabet{\mb}{OT1}{pzc}{m}{it}
\usepackage{amsaddr}
\usepackage{xcolor}
\usepackage{stmaryrd}
\usepackage{mathrsfs}
\begin{document}
\title[Inertial Hopf-cyclic homology]{Inertial  Hopf-cyclic homology}
\author{Tomasz Maszczyk}
\address{University of  Warsaw, 
 Institute of Mathematics,\\
ul.\ Banacha 2,
02--097 Warszawa, Poland}
\email{t.maszczyk@uw.edu.pl}
\author{Serkan  S\"utl\"u}
\address{I\c{s}{\i}k University, Department of Mathematics, \\ 34980, \c{S}ile, \.Istanbul, Turkey}
\email{serkan.sutlu@isikun.edu.tr}
\begin{abstract} We construct, study and apply a characteristic map from the relative periodic cyclic homology of the quotient map for a group action to the periodic Hopf-cyclic homology with coefficients associated with inertia of the action. This result admits, and in fact comes from, its noncommutative-geometric, or quantized, counterpart. 
The crucial ingredient is the construction of the appropriate quantization of the cyclic nerve of the action groupoid, the cyclic object related to inertia, as the Connes-cyclic dual of a Hopf-cyclic object with coefficients in some stable anti-Yetter--Drinfeld module quantizing the Brylinski space.
For the Hopf-Galois quantization of the case of trivial inertia, we find a non-trivial identification of our characteristic map with a well-known isomorphism of Jara--\c{S}tefan. In presence of nontrivial inertia, 
for an analytic ramified Galois double cover, we show that the cokernel of our characteristic map, which can be interpreted as  an invariant of inertia modulo topology of the maximal free action, is supported on the branch locus of the quotient map. Finally, we use our inertial Hopf-cyclic object to construct a new invariant of finite-dimensional algebras.   

\vspace{2em}

\end{abstract}

\maketitle
\pagebreak
\tableofcontents

\section{Introduction} 

Hopf-cyclic homology arose from studying the index theory of transversally elliptic operators \cite{ConnMosc98} to become an invariant of a noncommutative symmetry given in terms of a Hopf algebra. As such, it is a generalization of group and Lie algebra homology.  Its importance stems from the existence of characteristic maps relating it with cyclic homology, establishing a relation between noncommutative symmetry and noncommutative topology. 

Presented as a homology theory associated to a Hopf algebra and a pair of elements, consisting of a group-like and a character, called a modular pair in involution (MPI) in \cite{ConnMosc98}, Hopf-cyclic homology was axiomatized in \cite{HajaKhalRangSomm04-I,HajaKhalRangSomm04-II} as a homology theory associated to a Hopf algebra and a coefficient space, called a stable anti-Yetter-Drinfeld (SAYD) module. In the language of \cite{HajaKhalRangSomm04-II}, the MPI of \cite{ConnMosc98} represented a one dimensional coefficient space. 

However, a concrete nontrivial application of more general SAYD coefficient space had to wait until \cite{JaraStef06} of Jara and \c{S}tefan, in which a class of SAYD modules, a generalization of one of the examples from \cite{HajaKhalRangSomm04-I,HajaKhalRangSomm04-II}, arose in a canonical way within the context of the Hopf-Galois extensions. The application consisted there in establishing a cyclic duality between the (co)cyclic objects computing the relative cyclic homology and the Hopf-cyclic homology with appropriate SAYD coefficients associated with the Hopf-Galois symmetry of the extension. 

Next, in \cite{RangSutl}, further examples of MPI's were constructed for bicrossproduct Hopf algebras associated to Lie algebras. These MPI's were then upgraded into nontrivial SAYD modules in \cite{RangSutl-III}. More precisely, a 4-dimensional SAYD module over the Schwarzian quotient $\C{H}_{1S}$ of the Connes-Moscovici Hopf algebra $\C{H}_1$ was illustrated in \cite{RangSutl-III}. On the other hand, SAYD modules over the universal enveloping algebras were studied in \cite{RangSutl-II}, and a SAYD module was constructed for any given dimension; namely the truncated Weil algebras.

Each time a new situation was considered, a new type of SAYD coefficients was needed to allow the use of Hopf-cyclic homology. In the present paper we show how a Hopf-cyclic homology with appropriate SAYD coefficients quantizes the cyclic nerve of the action groupoid. We also show how to quantize the canonical internal functor from the action groupoid to the descent groupoid after passing to their cyclic nerves. In this context the prominent role is played by the Brylinski $G$-space, whose action groupoid is the inertia groupoid \cite{Moer02} of the action groupoid. The fundamental role of the Brylinski scheme relies on the fact that its category of equivariant sheaves of $\mathcal{O}$-modules is equivalent to the Drinfeld double of the monoidal category of equivariant sheaves of $\mathcal{O}$-modules on the $G$-scheme $X$ itself, at least for a finite group $G$, conjecturally in much wider generality \cite{Hin07}. Moreover, the inertia groupoid, through the so called \emph{extended quotient} construction, plays a role in 
 \emph{local Langlands duality} \cite{AubeBaumPlymSoll14}, as well as in the \emph{orbifold and stringy cohomology} \cite{ChenRuan04,LupeUrib04,BrylNist94}. Our construction can be understood as a vertical (along the quotient map) companion of the orbifold cohomology, when the orbifold is defined as a global quotient.  The quantization of the Brylinski $G$-space is provided by the following SAYD module. 
 
\noindent {\bf Theorem \ref{thm-SAYD}.}
{\it
For any right $\C{H}$-comodule algebra $A$ the quotient vector space 
\[
M:=(\C{H}\ot A)/\langle \mb{h}a'\ps{1}\ot aa'\ps{0}-\mb{h}\ot a'a \mid \mb{h}\in \C{H},\,a,a' \in A \rangle .
\]
is a stable anti-Yetter Drinfeld $\C{H}$-module with the left $\C{H}$-module structure 
\begin{align*}
\C{H}\ot M&\rightarrow M,\\
 \mb{h}'\ot [\mb{h}\ot a]&\mapsto \mb{h}'[\mb{h}\ot a]:=[\mb{h}'\mb{h}\ot a],
\end{align*} 
and the right $H$-comodule structure given by
\begin{align*}
M&\rightarrow M\ot \C{H} ,\\
 [\mb{h}\ot a]&\mapsto  [\mb{h}\ot a]\ns{0}\ot  [\mb{h}\ot a]\ns{1}:=[\mb{h}\ps{2} \ot a\ps{0}]\ot \mb{h}\ps{3} a\ps{1}S(\mb{h}\ps{1}).
\end{align*}
}

The obove SAYD module in the commutative case reproduces the affine Brylinski $G$-scheme. This generalizes our observation from \cite{MaszSutl14} that for a special case of the homogeneous quotient-coalgebra-Galois extensions the resulting cyclic object is related to inertia phenomena. We must stress, however, that in the Galois case, when the inertia is trivial, finding the appropriate Hopf-cyclic object with SAYD coefficients is a matter of transporting the cyclic object structure using an iterated invertible canonical map. In presence of inertia, when the canonical map is not invertible, finding a candidate for the Hopf-cyclic object beeing the receptacle for the map induced by iterating the canonical map is not trivial. Especially, the Hopf-Galois case cannot give any clue about the relation of the Hopf-cyclic homology with inertia. 
Our generalization to the noncommutative case is quite nontrivial for the presence of an additional Hopf algebra-type symmetry, necessary for the Hopf-cyclic approach. This cannot be read from the classical picture where this additional datum cancels out (see the last remark of \ref{bryl}). The appropriate noncommutative generalization of that cancellation is provided by  Proposition~\ref{basechange} about the Hopf algebra change for SAYD-modules and Hopf-algebra-change invariance of the inertial Hopf-cyclic object.  Showing that  the Connes-cyclic dual of a Hopf-cyclic object with coefficients in our stable anti-Yetter--Drinfeld module is an appropriate quantization of the cyclic nerve of the action groupoid, the cyclic object related to inertia, is the main conceptual novelty of the present paper.

Chronological development of the present idea goes back to \cite{MaszSutl14}, where we generalized a special case of the  Jara-\c{S}tefan isomorphism from homogeneous Hopf-Galois extensions to  homogeneous quotient coalgebra-Galois extensions, covering the case of quantum circle bundles over Podle\'{s} spheres \cite{Pod87, Brze96, Brze97}. In this generality, the thus obtained isomorphism is a cyclic-homological enhancement of the Takeuchi-Galois correspondence between the left coideal subalgebras, and the quotient right module coalgebras of a Hopf algebra \cite{MaszSutl14}. The latter generalizes the classical bijective correspondence between pointed orbits and stabilizers. 
Now, by constructing an explicit isomorphism, we show that in the Hopf-Galois case, when the module coalgebra is the Hopf-algebra itself, our (apparently different) construction produces the same SAYD-module as the one of Jara and \c{S}tefan~\cite{JaraStef06}.
Based on this, we generalize the Jara-\c{S}tefan isomorphism from Hopf-Galois extensions to quotient coalgebra-Galois extensions,  including the quantum instanton principal bundle~\cite{BoneCiccTarl02,BoneCiccDabrTarl04}. 
Together with the quantum circle bundles over Podle\'{s} spheres, they are important noncommutative deformations of principal bundles, which go beyond the Hopf-Galois context. Such a generality is required to quantize Poisson Hopf principal bundles, whose coisotropic Poisson-Galois symmetry is quantized to a module coalgebra-Galois symmetry \cite{BoneCiccTarl02, BoneCiccDabrTarl04}. Encompassing these examples requires a new approach in which we stress the necessity of separating the role of the \emph{inertia} of the Hopf algebra coaction from the \emph{Galois symmetry} viewed as the module coalgebra coaction. In the case of the quantum instanton bundle and the quantum circle bundles over Podle\'{s} spheres, the coaction of the Hopf algebra is transitive and hence is much bigger than the Galois symmetry given by the free coaction of the quotient module coalgebra, in opposite to the Hopf-Galois case, where these two symmetries coincide. By necessity forced by examples, the central role of inertia and its relation with Hopf-cyclic homology emerged inevitably.

It is worth mentioning that the comparison between our construction and the one of Jara-\c{S}tefan  \cite{JaraStef06}, which we provide in the present paper,  is quite nontrivial. The main difference consists of the construction of an appropriate SAYD module, which in \cite{JaraStef06} was almost dictated by the Miyashita-Ulbrich action,  absent in our generality. Although this obstacle could be walked around in the special case of the homogeneous quotient coalgebra-Galois extensions of \cite{MaszSutl14}, the (most of) generality discussed in the present paper requires a conceptually new structure of the SAYD module. 

The construction of our characteristic map goes in two steps. First, we construct an isomorphism between a cyclic object of the Takeuchi-Hopf algebroid-cyclic theory of a coring with SAYD coefficients in the comodule algebra itself, and the cyclic objects of the Hopf algebra-cyclic theory of a module coalgebra with coefficients in the aforementioned SAYD module $M$.
Next, we compose it with the map induced by the canonical map. The latter can be viewed as a  morphism of corings with group-like elements, where one of the corings encodes the descend data along the algebra extension, and the other  determines the symmetry given in terms of a module-coalgebra coaction on a comodule-algebra. The Galois condition means that it is an isomorphism of corings, but even in  presence of inertia the map makes sense. The cyclic-dual Hopf-cyclic homology of the first coring turns out to be isomorphic to the relative cyclic homology of the extension, and it depends only on the relative topology of the corresponding principal fibration (at least on the level of relative periodic cyclic homology, and in the classical commutative case). 
By the above isomorphism we started with in the first step, the Hopf-cyclic object of the second coring  turns out to be isomorphic to the Hopf-cyclic homology of the module coalgebra  coacting  on a comodule algebra.

In  the case of trivial inertia, our characteristic map generalizes the Jara-\c{S}tefan isomorphism from Hopf-Galois extensions to arbitrary module coalgebra-Galois extensions.
    
\noindent {\bf Theorem \ref{thm-HP}.}
{\it For any right $\C{H}$-module coalgebra $C$, and any $C$-Galois right $\C{H}$-comodule algebra extension $A$ of an algebra $B$, the relative periodic cyclic homology of the extension $B\subseteq~\!\!A$ is isomorphic to the Hopf-cyclic homology of the right $\C{H}$-module coalgebra $C$ with left-right SAYD coefficients in $M$. In short,
\[HP_{\bullet}(A\hspace{0.2em}|\hspace{0.1em}B)\cong HP^{\C{H}}_{\bullet}(C, M). \]
}

\vspace{-1.5em}
The case of nontrivial inertia is even more interesting. Beyond the $C$-Galois case, allowing non-free coactions of a right $\C{H}$-module coalgebra $C$ on a right $\C{H}$-comodule algebra extension $A$, we still have a map
\[HP_{\bullet}(A\hspace{0.2em}|\hspace{0.1em}B)\longrightarrow HP^{\C{H}}_{\bullet}(C, M). \]
It can be seen as a quantization of the \emph{ characteristic map} from topological invariants of the corresponding family of subgroup orbits to representation-theoretical invariants of a (sub)group action.

The discrepancy between topology and representation theory in the presence of inertia fenomena is a potential source of interesting invariants of equivariant geometry, coming from this map. For instance, we show that after analytification of a ramified Galois double cover, the cokernel of our characteristic map, which can be interpreted as  an invariant of inertia modulo topology of the maximal free action, is supported on the branch locus of the quotient map.

Finally, we apply our construction to associate with any finite-dimensional algebra its inertial Hopf-cyclic object of the universal Hopf algebra coaction, a new cyclic-homological invariant of finite dimensional algebras. Our characteristic map maps the periodic cyclic homology of a given finite-dimensional algebra to that universal inertial periodic Hopf-cyclic homology. We expect that the latter can distinguish algebras with the same periodic cyclic homology, a poor invariant of finite dimensional algebras, which by the Goodwillie theorem  \cite{good} is insensitive to nilpotent extensions.


\section{The classical  context to be quantized}
By a \emph{category of spaces} we  will mean a category with a final object and finite fiber products. In fact, we have in mind especially topological spaces, schemes over a field, complex algebraic and analytic spaces. The present chapter plays the role of a classical motivation for its more involved noncommutative-geometric, or quantum, generalization in the next chapter.      
\subsection{The cyclic nerve} Let us regard an internal category 
\[
\xymatrix{
 X \ar@<0ex>[r] & \mb{C}  \ar@<1ex>[l] \ar@<-1ex>[l] \ar@<1ex>[r] \ar@<-1ex>[r]& \mb{C}\times_{X}\mb{C} \ar@<2ex>[l] \ar@<0ex>[l] \ar@<-2ex>[l]
}
\]
in the category of spaces, with the space of objects $X$ and the space of morphisms  $\mb{C}$, as a truncated simplicial space with the \emph{source} and the \emph{target} maps $s, t: \mb{C}\rightarrow X$,  the \emph{composition} and the \emph{projection} maps $m, p, q:\mb{C}\times_{X}\mb{C}\rightarrow \mb{C}$, $m(c, \widetilde{c})=c\circ\widetilde{c}$, $p(c, \widetilde{c})=c$ and $q(c, \widetilde{c})=\widetilde{c}$  as face maps,  and the \emph{units} map $u: X\rightarrow \mb{C}$ and the  maps $u\times_{X}\mb{C}, \mb{C}\times_{X}u: \mb{C}\rightarrow\mb{C}\times_{X}\mb{C}$ as degeneracy maps, where the fiber product $\mb{C}\times_{X}\mb{C}$ is defined under the flipped span structure
\[
\xymatrix{
&  \mb{C} \ar@<0ex>[ld]_{t} \ar@<0ex>[rd]^{s} & \\
X & & X.
}
\]
In the present paper, we adopt the following inverse convention
\begin{equation}\xymatrix{
x_{0} & x_{1} \ar@<0ex>[l]_{c_{0}} & \cdots\ar@<0ex>[l]_{c_{1}} & x_{n} \ar@<0ex>[l]_{\ c_{n-1}}& x_{0} \ar@<0ex>[l]_{c_{n}}
}\label{cycle}
\end{equation}
to define the \emph{internal cyclic nerve} $CN_{\bullet}(\mb{C})$ of $\mb{C}$ where
\begin{align*}
CN_{n}(\mb{C})=
\left(\mb{C}^{\times_{X}(n + 1)}\right)
\times_{X\times X}X,
\end{align*}
being a cyclic space with respect to 
 the standard faces 
\begin{align}\label{nerve-faces}
\begin{gathered}  
d_{j}: CN_{n}(\mb{C}) \longrightarrow  CN_{n-1}(\mb{C}),\\
d_{j}\left(c_{0},\ldots, c_{n}\right) =\begin{cases} \left(c_{0},\ldots, c_{j}\circ c_{j+1}, \ldots c_{n}\right), & \mbox{if } \ 0\leq j\leq n-1 \\ \left(c_{n}\circ c_{0},\ldots,  c_{n-1}\right), & \mbox{if }\ i=n \end{cases} 
\end{gathered}     
\end{align}
degeneracies 
\begin{align}\label{nerve-degen}
\begin{gathered}  
s_{j}: CN_{n}(\mb{C}) \longrightarrow  CN_{n+1}(\mb{C}),\\
s_{j}\left(c_{0},\ldots, c_{n}\right) =\begin{cases} \left(c_{0},\ldots, c_{j}, {\rm id}_{x_{j+1}}, c_{j+1}, \ldots c_{n}\right), & \mbox{if } \ 0\leq j\leq n-1 \\ \left( c_{0},\ldots,  c_{n}, {\rm id}_{x_{0}}\right), & \mbox{if }\ i=n \end{cases}, 
\end{gathered}     
\end{align}
and the cyclic operator
\begin{align}\label{nerve-cyc}
\begin{gathered}  
t: CN_{n}(\mb{C}) \longrightarrow  CN_{n}(\mb{C}),\\
t\left(c_{0}, c_{1},\ldots, c_{n}\right) = \left(c_{n}, c_{0}, \ldots, c_{n-1}\right).
\end{gathered}     
\end{align}
The purpose of our convention is to reconcile  the category theory convention of composing morphisms with the convention of composing spans using fiber products and the right group action convention customary in the theory of principal bundles or the left group action on algebras (of functions) customary in Galois theory. 

\subsection{The inertial cyclic space as a cyclic nerve}
\subsubsection{The inertial cyclic space}\label{bryl}
For a given right $G$-space $X$ one defines a right $G$-space
\[
Bry_{G}(X):=\{ (g, x)\in Ad(G)\times X\ \mid\ xg=x\},
\] 
with the diagonal right $G$-action by right conjugations on $Ad(G)=G$ and a given right $G$-action on $X$, equipped with a $G$-equivariant morphism into $Ad(G)$ \cite{Bry87, Moer02}. It is called the \emph{Brylinski $G$-space}. The action groupoid of that $G$-space is the \emph{inertia groupoid} of the action groupoid of the initial $G$-space $X$. 

\noindent{\bf Remark.} In the case of a transitive right $G$-action on $X$ with a distinguished point $x_{0}\in X$,  we can assume that $X=\raisebox{-0.3ex}{\it K}\hspace{0.5ex}\backslash \raisebox{0.1ex}{\it G}$, the left quotient by the stabilizer $K$ of $x_{0}$. Then the Brylinski $G$-space is a right $G$-orbit  of $K=K\times \{Ke\}$ in $Ad(G)\times (\raisebox{-0.3ex}{\it K}\hspace{0.5ex}\backslash \raisebox{0.1ex}{\it G})$ under the diagonal right $G$-action. Moreover, in this case, the stabilizer of a point $(k, Ke)g$ of the Brylinski $G$-space in the inertia groupoid is equal to $g^{-1}Z_{K}(k)g$, the conjugate centralizer $Z_{K}(k)$ of $k$ in $K$.

\noindent{\bf Example.} An example of such a situation is provided by the Hopf fibration of $S^{7}$ over $S^{4}$ with the fibre $S^{3}$, also called \emph{the instanton principal bundle}, where $G=U(4)$, $H=SU(2)$, $K=U(3)$, $X=\raisebox{-0.3ex}{\it K}\hspace{0.5ex}\backslash \raisebox{0.1ex}{\it G}=\raisebox{-0.3ex}{{\it U}\! (3)}\hspace{0.1ex}\backslash \raisebox{0.1ex}{{\it U}\! (4)}= S^{7}$, $Y=\raisebox{0.1ex}{\it X}\slash \raisebox{-0.3ex}{\it H}=\raisebox{-0.3ex}{{\it U}\! (3)}\hspace{0.1ex}\backslash \raisebox{0.1ex}{{{\it U}\! (4)}}\slash \raisebox{-0.3ex}{{\it SU}\! (2)}=S^{4}$. 

For every right $G$-space $X$ and a closed subgroup $H<G$, we have a cyclic space~\cite{bo_ka_72} $I_{\bullet}^{H, G}(X)$ with
$$I_{n}^{H, G}(X) := H^{n+1}\times_{G}Bry_{G}(X),$$
where $H^{n+1}$ is the Cartesian product of copies of $H$, and the fibre product over $G$  is meant with respect to a pair of maps 
\begin{align}
\begin{split}
H^{n+1} &\longrightarrow G,\ \ \ 
(h_{0},\ldots,h_{n})\mapsto h_{0}\cdots h_{n},\\   
Bry_{G}(X) &\longrightarrow G,\ \ \ 
(g, x)\mapsto g,
\end{split}   
\end{align}
and the cyclic space structure is given by the faces 
\begin{align*}
\begin{gathered}
d_{j}: I_{n}^{H, G}(X) \longrightarrow I_{n-1}^{H, G}(X),\\
d_{j}\left((h_{0},\ldots, h_{n}), (g, x)\right) = \left((h_{0},\ldots, h_{j}h_{j+1}, \ldots h_{n}), (g, x)\right), 
\end{gathered}   
\end{align*}
for $0\leq j\leq n-1$, 
\begin{align*}
d_{n}\left((h_{0},\ldots, h_{n}), (g, x)\right) = \left((h_{n}h_{0}, \ldots h_{n-1}), (h_{n}gh_{n}^{-1}, xh_{n}^{-1})\right), 
\end{align*}
degeneracies 
\begin{align*}
\begin{gathered}
s_{j}: I_{n}^{H, G}(X) \longrightarrow I_{n+1}^{H, G}(X),\\ 
s_{j}\left((h_{0},\ldots, h_{n}), (g, x)\right) = \left((h_{0},\ldots, h_{j-1}, e, h_{j}, \ldots h_{n}), (g, x)\right),
\end{gathered}   
\end{align*}
for $0\leq j\leq n$, and the cyclic operator
\begin{align*}
\begin{gathered}
t: I_{n}^{H, G}(X)\longrightarrow I_{n}^{H, G}(X),\\
t\left((h_{0},\ldots, h_{n}), (g, x)\right) = \left((h_{n},  h_{0}\ldots, \ldots h_{n-1}), (h_{n}g h_{n}^{-1}, x h_{n}^{-1})\right).
\end{gathered}     
\end{align*} \color{black}

The cocyclic space structure of the cocyclic dual $I^{\bullet}_{H, G}(X)$ to cyclic space $I_{\bullet}^{H, G}(X)$ is given by the cofaces 
\begin{align*}
\begin{gathered}
\delta_{j}: I^{n}_{H, G}(X) \longrightarrow I^{n+1}_{H, G}(X),\\
\delta_{j}\left((h_{0},\ldots, h_{n}), (g, x)\right) = \left((h_{0},\ldots, h_{j-1}, e, h_{j}, \ldots h_{n}), (g, x)\right),
\end{gathered}   
\end{align*}
for $0\leq j\leq n$, the codegeneracies 
\begin{align*}
\begin{gathered}
\sigma_{j}: I^{n}_{H, G}(X) \longrightarrow I^{n-1}_{H, G}(X),\\
\sigma_{j}\left((h_{0},\ldots, h_{n}), (g, x)\right) = \left((h_{0},\ldots, h_{j}h_{j+1}, \ldots h_{n}), (g, x)\right),
\end{gathered}   
\end{align*}
for $0\leq j\leq n-1$,
\begin{align*} 
\sigma_{n}\left((h_{0},\ldots, h_{n}), (g, x)\right) = \left((h_{n}h_{0},\ldots,  h_{n-1}), (h_{n}gh_{n}^{-1}, xh_{n}^{-1})\right),
\end{align*}
and the cocyclic operator
\begin{align*}
\begin{gathered}
\tau: I^{n}_{H, G}(X)\longrightarrow I^{n}_{H, G}(X),\\
\tau\left((h_{0},\ldots, h_{n}), (g, x)\right) = \left((h_{1},\ldots, \ldots h_{n}, h_{0}), (h_{0}^{-1}g h_{0}, x h_{0})\right).
\end{gathered}     
\end{align*}

\paragraph{\bf Remark.} This cocyclic dual dualized again by passing to algebras of functions is again a cyclic object in the category of commutative algebras. The latter is the main construction we are to quantize with use of Hopf-cyclic homology.
 
\noindent{\bf Remark.} Since a point of $H^{n+1}\times_{G}Bry_{G}(X)$ can be written as
\begin{align*}
((h_{0},\ldots, h_{n}), (g,x))=((h_{0},\ldots, h_{n}), (h_{0}\cdots h_{n},x)),   
\end{align*} 
the (co)cyclic object $I_{\bullet}^{H, G}(X)$ is in fact independent of $G$,
\begin{align*}
H^{n+1}\times_{G}Bry_{G}(X)=H^{n+1}\times_{H}Bry_{H}(X).   
\end{align*}
Moreover, whenever $H$ acts on $X$ freely the latter simplifies even further 
\begin{align*}
H^{n+1}\times_{H}Bry_{H}(X)=(H^{n+1}\times_{H}\{pt\})\times X   
\end{align*}
where $\{pt\}\rightarrow H$ picks the neutral element.
This explains why in the classical commutative Hopf-Galois case one can  see neither the role of  the Brylinski scheme nor of the Hopf algebra $\C{H}=\mathcal{O}(G)$. All that makes a Hopf-entwined-coalgebra-Galois quantization nontrivial for the lack of hint that in fact the roles of the Hopf-type symmetry (the noncommutative counterpart of the $G$-action) and the module-coalgebra-type symmetry    
(the noncommutative counterpart of the $H$-action) are different. The former relates the Brylinski scheme to the Hopf-coaction part of the entwining, the latter corresponds to the Galois-type symmetry. 
\subsubsection{The cyclic nerve of the action groupoid}
Let $G$ be a topological group and $X$ be a right $G$-space. We will consider an action groupoid $H\times X$ of $X$ acted on by the subgroup $H$, with the right $H$-action 
\begin{align}
\begin{gathered}
X\times H\rightarrow X ,\\
\left(x, h\right)    \mapsto xh.
\end{gathered}
\end{align} 
It is equipped with the source and target maps
\begin{align}
\begin{gathered}
s, t: H\times X\rightarrow X ,\\
s\left(h, x\right) = x,\ \ t\left(h, x\right) = xh^{-1},
\end{gathered}
\end{align}
the units map
\begin{align}
\begin{gathered}
u: X\rightarrow H\times X ,\\
u\left( x\right) = \left(e, x\right),
\end{gathered}
\end{align}
and the composition
\begin{align}
\begin{gathered}
m: \left( H\times X\right)\times_{X}\left( H\times X\right)\rightarrow H\times X ,\\
\left( h, x\right)\circ\left(\widetilde{h}, \widetilde{x}\right) = \left(h\widetilde{h}, \widetilde{x}\right).
\end{gathered}
\end{align}
\begin{proposition}
For an action groupoid $H\times X$,  one has an isomorphism of cyclic spaces
\begin{align}\label{iso-cyc-act}
\begin{gathered}
CN_{\bullet}(H\times X)\rightarrow I_{\bullet}^{H, G}(X),\\
\left( (h_{0}, x_{0}), \ldots, (h_{n}, x_{n})\right) \mapsto \left((h_{1}, \ldots, h_{n}, h_{0}), (h_{1}\ldots h_{n}h_{0}, x_{1}\right).
\end{gathered}
\end{align}  
\end{proposition}
\begin{proof}Checking the preservation of the cyclic structure by the map  $\ref{iso-cyc-act}$ is straightforward. It is enough to observe that the inverse to $\ref{iso-cyc-act}$ is provided by the well defined map 
\begin{align}\label{I-CN}
\begin{gathered}
CN_{\bullet}(H\times X) \leftarrow I_{\bullet}^{H, G}(X),\\
\left( (h_{n}, x_{n}), (h_{0}, x_{0}), \ldots, (h_{n-1}, x_{n-1})\right)\mapsfrom  \left((h_{0}, \ldots, h_{n}), (h_{0}\ldots h_{n}, x\right)) 
\end{gathered}
\end{align}
where $x_{0}:=x$ and $x_{i}:=xh_{0}\ldots h_{i-1}$ for $i>0$.
\end{proof}
\subsection{The Alexander-Spanier cyclic space as a cyclic nerve} 
\subsubsection{The Alexander-Spanier cyclic space} For every continuous (resp. regular, etc.) map $X\rightarrow Y$, we have a standard cyclic space $Y^{X}_{\bullet }$,  with $Y^{X}_{n}$ being the fibre product of $n + 1$ copies of a space $X$ over $Y$, equipped with the standard faces 
\begin{align}\label{AS-faces}
\begin{gathered}  
d_{j}: Y^{X}_{n } \longrightarrow  Y^{X}_{n-1},\\
d_{j}\left(x_{0},\ldots, x_{n}\right) =
\left(x_{0},\ldots, x_{j-1},  x_{j+1}, \ldots x_{n}\right) \ \  \mbox{for} \ 0\leq j\leq n,  
\end{gathered}     
\end{align}
degeneracies 
\begin{align}\label{AS-degen}
\begin{gathered}  
s_{j}: Y^{X}_{n } \longrightarrow  Y^{X}_{n +1},\\
s_{j}\left(x_{0},\ldots, x_{n}\right) = \left(x_{0},\ldots, x_{j},  x_{j}, \ldots x_{n}\right), \ \  \mbox{for} \ 0\leq j\leq n,
\end{gathered}     
\end{align}
and the cyclic operator
\begin{align}\label{AS-cyc}
\begin{gathered}  
t: Y^{X}_{n} \longrightarrow  Y^{X}_{n},\\
t\left(x_{0}, x_{1},\ldots, x_{n}\right) = \left(x_{n}, x_{0}, \ldots, x_{n-1}\right).
\end{gathered}     
\end{align}
\paragraph{\bf Remark.} We call the cyclic space $Y^{X}_{\bullet}$ Alexander-Spanier referring to the well known fact that, for $Y$ being a one point space, its simplicial part of the cyclic structure induces on spaces of function germs along the diagonals a structure of the standard complex computing the Alexander-Spanier cohomology of $X$.

By $X_{Y}^{\bullet}$  we will denote the cocyclic space whose  Connes-cyclic-dual is the Alexander-Spanier cyclic space $Y^{X}_{\bullet}$. Explicitly, it is equipped with the cofaces 
\begin{align*}
\begin{gathered}  
\delta^{j}: X_{Y}^{n} \longrightarrow  X_{Y}^{n+1},\\
\delta^{j}\left(x_{0},\ldots, x_{n}\right) = \begin{cases}\left(x_{0},\ldots, x_{j},  x_{j}, \ldots x_{n}\right), & \ \  \mbox{if} \ 0\leq j\leq n,\\
\left(x_{0}, x_{1}\ldots, x_{n},  x_{0}\right), & \ \  \mbox{if} \  j= n+1,
\end{cases}
\end{gathered}     
\end{align*}
the codegeneracies 
\begin{align*}
\begin{gathered}  
\sigma^{j}: X_{Y}^{n} \longrightarrow  X_{Y}^{n -1},\\
\sigma^{j}\left(x_{0},\ldots, x_{n}\right) = \left(x_{0},\ldots, x_{j},  x_{j+2}, \ldots x_{n}\right),
\end{gathered}     
\end{align*}
for $0\leq j\leq n-1$, and the cocyclic operator
\begin{align*}
\begin{gathered}  
\tau: X_{Y}^{n} \longrightarrow  X_{Y}^{n},\\
\tau\left(x_{0}, x_{1},\ldots, x_{n}\right) = \left(x_{1},  \ldots, x_{n}, x_{0},\right).
\end{gathered}     
\end{align*}
Note that for  $Y$ being a point and $X$ being a compact smooth  manifold, the periodic cyclic homology of the cyclic space of smooth functions on this cocyclic manifold computes a $\mathbb{Z}/ 2\mathbb{Z}$-graded version of the deRham cohomology of $X$ \cite{Conn83}.
\subsubsection{An isomorphism to a cyclic nerve} 
\begin{proposition} A continuous map $X\rightarrow Y$ equalizing the pair $s, t:  \mb{C}\rightarrow X$
\[
\xymatrix{
Y &  X \ar@<0ex>[l] & \mb{C}  \ar@<0.5ex>[l] \ar@<-0.5ex>[l] 
}
\]      
induces a map of cyclic spaces
\begin{align}
\begin{gathered}
CN_{\bullet}(\mb{C})\rightarrow Y^{X}_{\bullet},\\
\left(c_{0}, c_{1},\ldots, c_{n}\right)\mapsto \left(x_{1}, \ldots, x_{n}, x_{0}\right)
\end{gathered}
\end{align}
for any cycle as in \ref{cycle}.
\end{proposition}
\begin{proof}It is a straightforward consequence of the definitions \ref{nerve-faces} - \ref{nerve-cyc} and \ref{AS-faces} - \ref{AS-cyc} 
of the cyclic structures. \end{proof}


\begin{proposition}\label{alex-pair}
For a pair groupoid $X\times_{Y}\hspace{-0.2em}X$, one has an isomorphism of cyclic spaces
\begin{align}\label{iso-cyc-pair}
\begin{gathered}
CN_{\bullet}(X\times_{Y}\hspace{-0.2em}X)\rightarrow Y^{X}_{\bullet},\\
\left( (x_{0}, x_{1}), (x_{1}, x_{2}),\ldots, (x_{n-1}, x_{n}), (x_{n}, x_{0})\right) \mapsto \left(x_{1}, \ldots, x_{n}, x_{0}\right).
\end{gathered}
\end{align}  
\end{proposition}
\begin{proof}Checking the preservation of the cyclic structure by  $\ref{iso-cyc-pair}$ is straightforward. It is enough to observe that the inverse to $\ref{iso-cyc-pair}$ is provided by the well defined map 
\begin{align}\label{AS-CN}
\begin{gathered}
CN_{\bullet}(X\times_{Y}\hspace{-0.2em}X) \leftarrow Y^{X}_{\bullet},\\
\left( (x_{n}, x_{0}), (x_{0}, x_{1}),\ldots, (x_{n-1}, x_{n})\right)\mapsfrom \left(x_{0}, \ldots, x_{n}\right).
\end{gathered}
\end{align}
\end{proof}
\subsection{The characteristic map as a functor-induced map of cyclic nerves} For any group homomorphism $H\rightarrow G$, a right $G$-space $X$, and an $H$-invariant map $X\rightarrow Y$ we have a canonical internal functor of groupoids
\begin{align}\label{int-fun}
\begin{gathered}
H\times X \rightarrow X\times_{Y}\hspace{-0.2em}X,\\
(h, x)\mapsto (xh^{-1}, x)
\end{gathered}
\end{align} 
which induces a map of internal cyclic nerves
\begin{align}\label{map-cyc-ner}
CN_{\bullet}(H\times X)\rightarrow CN_{\bullet}(X\times_{Y}\hspace{-0.2em}X).
\end{align}
We are going to relate \eqref{map-cyc-ner} with  the following \emph{characteristic map} of cyclic spaces.
\begin{proposition}\label{inert-nerve} 
Assuming that $X$ is a right $G$-space, and  the map $X\rightarrow Y$ is $H$-invariant with respect to a group homomorphism $H\rightarrow G$, we have  a morphism of cyclic spaces 
\begin{align}\label{char-on-points}
\begin{gathered}
\pi_{\bullet}: I_{\bullet}^{H, G}(X)\longrightarrow Y^{X}_{\bullet}, \\
((h_{0},\ldots, h_{n}), (g,x))\  \mapsto\  (x, xh_{0},xh_{0}h_{1},\ldots, xh_{0}h_{1}\cdots h_{n-1}).
\end{gathered}
\end{align}
\end{proposition}
\begin{proof}Preservation of the cyclic structure by the map (\ref{char-on-points}) can be easily checked. \end{proof}
 
\paragraph{\bf Remark.} It is also  easy to see that, provided $H$ acts transitively along the fibers of $X\rightarrow Y$, the morphism (\ref{char-on-points}) of cyclic spaces is degree-wise surjective. The fibre over a point $(x_{0},\dots, x_{n})\in Y^{X}_{n}$  on the right hand side, regarded as a cyclic sequence, can be identified then with a space of all cyclic sequences $(h_{0},\dots, h_{n})\in H^{n+1}$ 
such that $x_{j}h_{j}=x_{j+1}$, where the indices are counted modulo $n+1$.
That fibre is a torsor over $H_{x_{0}}\times\cdots\times H_{x_{n}}$ with respect to the left action
\begin{align}
\begin{split}
(\chi_{0},\dots, \chi_{n})(h_{0},\dots, h_{n})=(\chi_{0}h_{0},\dots, \chi_{n}h_{n})
\end{split}
\end{align}
of $(\chi_{0},\dots, \chi_{n})\in H_{x_{0}}\times\cdots\times H_{x_{n}}$. It then follows from the equality  $H_{x_{j+1}} =h_{j}^{-1}H_{x_{j}}h_{j}$ that the stabilizers $H_{x_{j}}$ of points in the cycle are all conjugate. As a result, the geometric fibre  over $(x_{0},\dots, x_{n})$ can be noncanonically identified with  $H_{x_{0}}^{n+1}$. In particular, if $H$ acts freely on $X$ the map (\ref{char-on-points}) becomes an isomorphism of cyclic spaces, with an explicit inverse
\begin{align}\label{inverse-on-points}
\begin{gathered}
I_{n}^{H, G}(X)\longleftarrow Y^{X}_{n}, \\
((x_{0}^{-1}x_{1}, x_{1}^{-1}x_{2}, \ldots, x_{n}^{-1}x_{0}), (e,x)) \  \mapsfrom\  (x_{0},\ldots, x_{n}),
\end{gathered}
\end{align}
where  on the left hand side the translation map 
\begin{align}\label{trans}
\begin{split}
X\times_{Y}X&\longrightarrow H, \\
(x_{0}, x_{1})& \  \mapsto\  x_{0}^{-1}x_{1},
\end{split}
\end{align}
of an $H$-principal bundle $X\rightarrow Y=X/H$ is used.
\subsubsection{The isomorphism of cyclic maps}
\begin{proposition}\label{iso-cyc-inert}  Under assumptions of Proposition~\ref{inert-nerve} one has a commutative diagram of cyclic spaces
\begin{align}\label{iso-cyc-fun}
\begin{array}{ccc}
I_{\bullet}^{H, G}(X) &\longrightarrow & CN_{\bullet}(H\times X),\\
\downarrow&   &\downarrow\\
Y^{X}_{\bullet} &\longrightarrow & CN_{\bullet}(X\times_{Y}\hspace{-0.2em}X),
\end{array}
\end{align}
where the left vertical map is the characteristic map (\ref{char-on-points}), the right vertical map is \eqref{map-cyc-ner}, the top horizontal arrow is (\ref{I-CN}) and the bottom horizontal arrow is (\ref{AS-CN}).
\end{proposition} 
\begin{proof}
Under the aforementioned internal functor  inducing a morphism of cyclic spaces between corresponding cyclic nerves, a cycle 
\begin{align}
(c_{0},\ldots, c_{n})=((h_{0},x_{0}),\ldots, (h_{n},x_{n})),
\end{align}
in the cyclic nerve of $\mb{C}=H\times X$, where we have 
\begin{align}\label{act-cyc}
x_{i+1}=x_{i}h_{i}
\end{align} in the cyclic order, maps into a cycle
\begin{align}
((x_{0},x_{1}),\ldots, (x_{n-1},x_{n}), (x_{n}, x_{0})).
\end{align}
Therefore, by (\ref{iso-cyc-pair}), (\ref{iso-cyc-act}) and (\ref{char-on-points}), commutativity of (\ref{iso-cyc-fun}) reads as the identity
\begin{align}
(x_{1},x_{1}h_{1}, x_{1}h_{1}h_{2},\ldots, x_{1}h_{1}\cdots h_{n-1}, x_{1}h_{1}\cdots h_{n})=(x_{1},\ldots, x_{n}, x_{0})
\end{align}
which holds by (\ref{act-cyc}) and the fact that $x_{0}h_{0}\cdots h_{n}=x_{0}$.
\end{proof}
\begin{corollary}
The commutative square \eqref{iso-cyc-fun} is an isomorphism between the characteristic map \eqref{char-on-points} and the map \eqref{map-cyc-ner} of internal cyclic nerves induced by the internal functor \eqref{int-fun}.
\end{corollary}
\begin{proof}
Since the horizontal maps in the commutative square \eqref{iso-cyc-fun} are isomorphisms, it provides an isomorphism of the vertical arrows.
\end{proof}
\begin{corollary}
Provided the action of a subgroup $H$ on $X$ is free and $X\rightarrow Y=X/H$ is a quotient map, the characteristic map is an isomorpism of cyclic spaces.
\end{corollary}
\begin{proof}
Under the assumptions the functor $\mb{C}=H\times X\rightarrow X\times_{Y}\hspace{-0.2em}X$ is an equivalence, hence it induces an isomorphism of corresponding cyclic nerves. By commutativity of (\ref{iso-cyc-fun}), with three arrows being isomorphisms of  cyclic spaces, the remaining left vertical arrow, the characteristic map, must be an isomorphism as well.
\end{proof}
\section{The Hopf-cyclic quantization}\label{Extensions and Cyclic Homology}
In this section we recall the material we shall use in the sequel. We begin with the notion of entwined coalgebra-Galois extension, which shall be followed by the Hopf-cyclic homology of corings, with coefficients. Finally, we shall recall the cyclic homology of algebra extensions.
\subsection{Module coalgebra-equivariant algebra extensions}
We recall the definition and basic properties of different coalgebra-equivariant extensions of algebras from \cite{BrzeHaja09}.
\subsubsection{Entwined algebra extensions}\label{Entwinings} Given a co-augmented coalgebra $C$, i.e. equipped with a group-like element $e\in C$, one can consider a kind of symmetry imposed on an algebra $A$ by an entwining $\psi: C\otimes A\longrightarrow A\otimes C$. This means that $\psi$ satisfies the identities tantamount to a choice of a \emph{mixed distributive law} between the comonad $C\ot (-)$ and the monad $A\ot (-)$ on the category of vector spaces, whose (co)monad structure comes from the (co)algebra structure. It is well known \cite{BrzeHaja09} that the map
\[A\rightarrow A\ot C,\ \ \ a\mapsto \psi(e\ot a)\]
makes $A$ a right $C$-comodule. The compatibility of that $C$-coaction and the multiplication and the unit of $A$ is encoded by $\psi$. In the consequence, the subspace 
\[B=A^{{\rm co}\hspace{0.125em} C}:=\{ b\in A \mid\ \psi(e\ot b)=b\ot e\} \]
of $A$ is a subalgebra, called \emph{algebra of invariants}. In this context,  $A$ is called an \emph{entwined $C$-algebra extension} of $B$.
\subsubsection{Module coalgebra-algebra extensions}\label{Module coalgebra-algebra extensions} It is the main class of extensions we are interested in. They require two kinds of symmetry, the Hopf algebra action and coaction, used for defining the entwining between a module coalgebra and a comodule algebra. The presence of the group-like element of the coalgebra, allows one to relate the Hopf algebra and the module coalgebra by a module coalgebra map under which their coactions on the algebra extensions become compatible.   Moreover, the Hopf algebra-type symmetry is necessary for the Hopf-cyclic quantization of the classical  characteristic map of cyclic spaces. Module coalgebra-algebra extensions are defined as follows. If  $A$ is  a right $H$-comodule algebra with the coaction
\[
A\lra A\ot \C{H},\ \ \ a\mapsto a\ps{0}\ot a\ps{1}
\]
and $C$ be a right  $\C{H}$-module coalgebra with a group-like element $e \in C$, such that, the induced coaction
\[
A\lra A\ot C,\ \ \ a\mapsto a\ns{0}\ot a\ns{1}:= a\ps{0}\ot e a\ps{1}
\]
has the subalgebra of invariants $B=A^{{\rm co}\ \!C}$, then we say that $A$ is a \emph{module coalgebra-algebra extension} of $B$. It is well known that it is a special case of an entwined algebra extension with the entwining
\begin{align}\label{stand-entw}
\psi: C\otimes A\longrightarrow A\otimes C, \ \ c\otimes a\mapsto 
a\ps{0}\ot ca\ps{1}.
\end{align}
If a right 
$\mathcal{H}$-module coalgebra $C$ and a right 
$\mathcal{H}$-comodule algebra $A$ are entwined by \eqref{stand-entw} we call them 
\emph{$\mathcal{H}$-entwined}.
\subsubsection{Quotient module coalgebra-comodule algebra extensions} An interesting, due to the geometric examples it covers, class of coalgebra-entwined algebra extensions are the \emph{quotient module coalgebra-comodule algebra extensions}. In this setting, one lets $\C{H}$ to be a Hopf algebra, $\C{I} \subseteq \C{H}$ a coideal right ideal, and $A$ a right $H$-comodule algebra via $\nb:A \to A \ot \C{H}$. Then the composition
\begin{equation*}
\xymatrix{
\rho: A \ar[r]^{\nb}& A \ot \C{H} \ar[r]^{A \ot \pi\,\,\,\,\,\,} & A \ot \C{H}/\C{I},
}
\end{equation*}
where $\pi$ denotes the canonical projection
\begin{equation*}
\pi: \C{H}\rightarrow \C{H}/\C{I},\ \ \ \mb{h} \mapsto  \wbar{\mb{h}},
\end{equation*}
determines a right $\C{H}/\C{I}$-coaction on $A$. Finally, a $\C{H}/\C{I}$-entwined algebra extension $A^{{\rm co}\,\C{H}/\C{I}} \subseteq A$ is called  \emph{quotient coalgebra-entwined extension}. 

\subsubsection{Coalgebra-Galois extensions}\label{Coalgebra-Galois extensions} 
An $C$-entwined algebra extension $B\subseteq A$ 
is said to be \emph{Galois} if the left $A$-linear right $C$-colinear map
\begin{equation}\label{aux-canonical-map}
can:A \ot_B A \lra A \ot C, \quad a \ot_B a' \mapsto 
a\psi(e\ot a')
\end{equation}
is bijective. The map \eqref{aux-canonical-map} is called the \emph{canonical map} of the entwined $C$-extension. If the canonical map for an entwined pair $(C, A, \psi)$ invertible, we also say that  $(C, A, \psi)$ satisfies \emph{the Galois condition}.

An interesting consequence of a entwined $C$-extension being Galois  is the fact that the entwining $\psi$ is uniquely determined by the $C$-comodule structure of $A$ by the explicit formula \cite{BrzeHaja09}
\[\psi(c\ot a)=can(can^{-1}(1\ot c)a).\]
Therefore, under Galois condition making sense for merely a coalgebra $C$ coaction on an algebra $A$, ``entwined'' is superfluous, hence the name \emph{$C$-Galois extension} used in \cite{BrzeHaja09}. However, since we go beyond the Galois case, we started with the most general notion of \emph{entwined $C$-algebra extension}, embracing all cases of interest. 

In the case when $\C{I} = 0$, the quotient coalgebra-Galois extensions recover the \emph{Hopf-Galois extensions}, and in the case when $A = \C{H}$ they are called \emph{homogeneous coalgebra-Galois extensions}. In particular, viewing $\C{H}$ as a right $\C{H}$-comodule algebra via its comultiplication, one obtains the homogeneous $\C{H}/\C{I}$-Galois extensions 
\begin{equation*}
\xymatrix{
\rho: \C{H} \ar[r]^{\D}& \C{H} \ot \C{H} \ar[r]^{\Id \ot \pi\,\,\,\,\,\,} & \C{H} \ot \C{H}/\C{I}, & \mb{h} \mapsto \mb{h}\ps{1} \ot \wbar{\mb{h}\ps{2}}.
}
\end{equation*}
Such Galois extensions correspond to the quantum homogeneous spaces. Concrete examples include Manin's plane and Podle\'s spheres \cite{Brze96}, as well as a quantum spherical fibration of $SU_q(2)$ \cite{Brze97}. The quantum instanton bundle introduced in \cite{BoneCiccTarl02} is also a quotient coalgebra-Galois extension, \cite{BoneCiccDabrTarl04}. The latter example requires the full generality of the quotient coalgebra-Galois extensions: $A \neq \C{H}$ and $\C{I} \neq 0$.

The following lemma shows that also ``quotient'' is superfluous in the notion of \emph{quotient module coalgebra-Galois algebra extension}  since  for module coalgebra-algebra extensions it follows from the Galois condition.

\begin{lemma} Let $\C{H}$ be a Hopf algebra, and $B\subseteq A$ be a module coalgebra-Galois algebra extension for an $\C{H}$-module coalgebra $C$ coaugmented with a group-like $e\in C$. Then, the right $\C{H}$-module coalgebra map 
$$\pi : \C{H}\rightarrow C,\ \ \ \ \mb{h}\mapsto e\mb{h}$$ is surjective, hence makes $C$ a quotient coaugmented right $\C{H}$-module coalgebra. Moreover, in this case, $B\subseteq A$ is a quotient module coalgebra-Galois extension.
\end{lemma}

\begin{proof} Consider a diagram commuting by virtue of the definition of a module coalgebra-comodule algebra extension
$$\xymatrix{
A\ot A \ar@{->>}[d] \ar[r]^{\beta} &A\ot \C{H}\ar[d]^{A\otimes \pi}\\
A\ot_B A \ar[r]^{can} &A\ot C.}$$
Since the left vertical arrow (the balancing the tensor product over $B$) is surjective  and the canonical map is bijective, the map $A\ot \pi$ is surjective as well. Since  $A$  is (faithfully) flat over the base field we have  $\ker (A\ot \pi)=A\ot \ker(\pi)$ and hence by exactness of the short sequence 
$$\xymatrix{
0\ar[r] & A\ot \ker(\pi)\ar[r] & A\ot \C{H}\ar[r]^{A\ot \pi}& A\ot C\ar[r]& 0}$$
and faithfull flatness of $A$ the map $\pi$ itself has to be  surjective.   
\end{proof}

Given a $C$-Galois extension $B\subseteq A$, there is also an $A$-coring structure on $A\ot C$, called the Galois-coring of the extension, whose underlying $A$-bimodule structure reads as
\[
a'(a\ot c)=a'a\ot c,\ \ \ (a\ot c)a'=aa'\ps{0}\ot ca'\ps{1}.
\]
This can be understood as the coincidence of the canonical entwining  
\[
c\ot a \mapsto can(can^{-1}(1\ot c)a)
\]
of  the $C$-Galois extension $A$, with the 
Hopf-entwining 
\[
c\ot a \mapsto a\ps{0}\ot c a\ps{1}
\]
for a right $\C{H}$-module coalgebra and a right $\C{H}$-comodule algebra $A$, \cite{BrzeHaja09,BrzeWisb-book}.

The fact that the canonical map (\ref{aux-canonical-map})
is an isomorphism of corings will play a  crucial role later on. Note that $A\ot_B A$, called the \emph{Sweedler coring},  encodes the relative geometry of a noncommutative fibration in terms of the descent data, while $A\ot C$, called   the \emph{Galois coring}, encodes the action of the symmetry given in terms of the coaction of the coalgebra~\cite{Caen04}.
 
Note that in the Hopf-entwined case the canonical map \eqref{aux-canonical-map} reads as
\begin{align}\label{entw-can}
can(a\ot_{B}a')=aa'\ps{0}\ot ea'\ps{1}
\end{align}
The main aim of the present paper is to use this map to cook up a  morphism of cyclic objects providing invariants of the relative topology of a 
symmetry-invariant fibration, on the \emph{topological} side, and invariants of the action of symmetry, on the \emph{dynamical}  side.

\subsection{Hopf-cyclic homology of corings}\label{Hopf-cyclic homology of corings}

We now recall the cyclic homology of a module coring, with coefficients in a stable-anti-Yetter-Drinfeld (SAYD) module, under the symmetry of a right $\times$-Hopf algebroid from \cite{BohmStef08,HassRang13}.

Let $A$ and $\bm{H}$ be two algebras. Let also $s:A\to \bm{H}$ and $t:A^{\rm op}\to \bm{H}$ be two algebra maps with commuting ranges. Then $\bm{H}$ is equipped with an $A$-bimodule structure by
\[
a \cdot \bm{h} \cdot a' = \bm{h} s(a')t(a)
\]
for any $a, a'\in A$ and any $\bm{h}\in \bm{H}$.  Let, in addition, $\bm{H}$ be an $A$-coring with the comultiplication $\D:\bm{H}\to \bm{H}\ot_{A} \bm{H}$, and the counit $\ve:\bm{H}\to A$. Then $(\bm{H},s,t,\D,\ve)$ is called a right $A$-bialgebroid if  for $\bm{h}, \bm{h}'\in \bm{H}$ and $a\in A$
\begin{itemize}
\item[(i)]  $\bm{h}\ps{1}\ot_A t(a)\bm{h}\ps{2}=s(a)\bm{h}\ps{1}\ot_A \bm{h}\ps{2}$,
\item[(ii)] $ \D(1_{\bm{H}})=1_{\bm{H}}\ot_A 1_{\bm{H}}, \quad \D(\bm{h}\bm{h}')=\bm{h}\ps{1}\bm{h}'\ps{1}\ot_A \bm{h}\ps{2}\bm{h}'\ps{2}$,
\item[(iii)] $ \ve(1_{\bm{H}})=1_A, \quad  \ve(\bm{h}\bm{h}')=\ve(s(\ve(\bm{h}))\bm{h}')$.
\end{itemize}
Finally, a right $A$-bialgebroid $\bm{H}=(\bm{H},s,t,\D,\ve)$ is said to be a right $\times_{A}$-Hopf algebra provided the map
\begin{equation*}
\nu: \bm{H}\ot_{A^{\rm op}}\bm{H}\lra \bm{H}\ot_{A}\bm{H}, \quad  \bm{h}\otimes_{A^{\rm op}}\bm{h}'\mapsto \bm{h}\bm{h}'\ps{1}\ot_{A}\bm{h}'\ps{2}
\end{equation*}
is bijective. As for the inverse, we use the notation
\begin{equation*}
\nu^{-1}(1\otimes_{A^{\rm op}}\bm{h})=\bm{h}^{-}\otimes_{A}\bm{h}^{+}
\end{equation*}
for any $\bm{h} \in \bm{H}$. For further information on bialgebroids and $\times$-Hopf algebroids we refer the reader to \cite{Bohm09,BrzeWisb-book,Scha98,Scha00}, and we recall here \cite[Ex. 2.3]{HassRang13}.
\begin{example}\label{example-env-alg-right-cross-Hopf-alg}{\rm
For an algebra $A$, its enveloping algebra $A^{\rm e}:=A\ot A^{\rm op}$ is a right $\times_A$-Hopf algebroid with the structure maps
\begin{align*}
& s: A\lra A^{\rm e}, \quad a\mapsto a\ot 1; \quad t: A^{\rm op}\lra A^{\rm e}, \quad a\mapsto 1\ot a,\\
& \D^{\rm e}:A^{\rm e}\lra A^{\rm e}\ot_A \hspace{-0.15em}  A^{\rm e}, \quad (a\ot \widetilde{a})\mapsto (1\ot \widetilde{a})\ot_A (a\ot 1),\\
&  \ve^{\rm e}:A^{\rm e}\lra A, \quad \ve^{\rm e}(a\ot \widetilde{a})=\widetilde{a}a,\\
& \nu((a\ot \widetilde{a})\ot_{A^{\rm op}} (a'\ot \widetilde{a}'))= (a\ot \widetilde{a}'\widetilde{a})\ot_A (a'\ot 1),\\
& \nu^{-1}((a\ot \widetilde{a})\ot_A (a'\ot \widetilde{a}'))= (a\widetilde{a}'\ot \widetilde{a})\ot_{A^{\rm op}} (a'\ot 1).
\end{align*}
}\end{example}
Now we recall SAYD modules over a right $\times_A$-Hopf algebroids from \cite{BohmStef08, HassRang13}. Let $\bm{H}$ be a right $\times_A$-Hopf algebroid, and $\bm{M}$ be a left $\bm{H}$-module and a right $\bm{H}$-comodule via $\nb:\bm{M} \to \bm{M} \ot_A \bm{H}$ defined by $\bm{m} \to \bm{m}\ns{0} \ot_A \bm{m}\ns{1}$. Then $\bm{M}$ is called an AYD module over $\bm{H}$ if the $A$-bimodule structures on $\bm{M}$, being a left module over a right $\times_A$-Hopf algebroid and a left comodule over an $A$-coring, coincide, \ie
\begin{equation}\label{aux-AYD-compatibility-cross-Hopf-I}
\bm{m} \cdot a = t(a)\rt \bm{m}, \quad a \cdot \bm{m} = \bm{m}\ns{0} \cdot \ve(s(a)\bm{m}\ns{1}) = s(a)\rt \bm{m},
\end{equation}
and
\begin{equation}\label{aux-AYD-compatibility-cross-Hopf-II}
\nb(\mb{h}\rt \bm{m})= {\mb{h}\ps{1}}^{+}\rt  \bm{m}\ns{0}\ot_{A} \mb{h}\ps{2} \bm{m}\ns{1} {\mb{h}\ps{1}}^{-}.
\end{equation}
Furthermore, $\bm{M}$ is called \emph{stable} if $\bm{m}\ns{1}\rt \bm{m}\ns{0}=\bm{m}$. The following example will be used later.
\begin{example}\label{example-SAYD-on-A-env}{\rm
Let $A$ be an algebra and $A^{\rm e}$ be the enveloping algebra of $A$ as the right $\times_A$-Hopf algebroid introduced in Example \ref{example-env-alg-right-cross-Hopf-alg}. Then $A$ is a left $A^{\rm e}$-module via
\[
A^{\rm e} \ot A \lra A, \quad (a'\ot \widetilde{a}') \rt a:= a'a\widetilde{a}',
\]
for all $a,a',a'' \in A$. Moreover, $A$ is a right $A^{\rm e}$-comodule via
\[
\nb :A \lra A \ot_A\hspace{-0.175em}  A^{\rm e}, \quad  a \mapsto  a\ns{0} \ot_A a\ns{1}:=1 \ot_A (a\ot 1).
\]
For any $a \in A$ we have
\[
a\ns{1} \rt a\ns{0} = (a\ot 1) \rt 1 = a,
\]
hence $A$ is stable. As for the AYD compatibility, for any $a, a', \widetilde{a}' \in A$ we have
\[
a\widetilde{a}' = (1\ot \widetilde{a}') \rt a =  t(\widetilde{a}') \rt a,
\]
as well as
\[
a\ns{0} \ve(s(a')a\ns{1}) = \ve((a'\ot 1)(a\ot 1))  = a'a = (a'\ot 1) \rt a =  s(a') \rt a.
\]
As a result, \eqref{aux-AYD-compatibility-cross-Hopf-I} is satisfied. Finally, 
\begin{align*}
 \nb((a'\ot \widetilde{a}') \rt a) & = 1 \ot_A (a'a\widetilde{a}'\ot 1) =
 (1\ot 1) \rt 1 \ot_A (a'\ot 1)(a\ot 1)(\widetilde{a}'\ot 1)\\
& = (1\ot \widetilde{a}')^+ \rt 1 \ot_A (a'\ot 1)(a\ot 1)(1,\widetilde{a}')^- \\
& = {(a'\ot \widetilde{a}')\ps{1}}^+ \rt a\ns{0} \ot_A (a'\ot \widetilde{a}')\ps{2}a\ns{1}{(a'\ot \widetilde{a}')\ps{1}}^- .
\end{align*}
Therefore \eqref{aux-AYD-compatibility-cross-Hopf-II} is also satisfied, and we conclude that $A$ is a left-right SAYD module over the right $\times_A$-Hopf algebroid $A^{\rm e}$.
}\end{example}

Let us next recall the cyclic homology of module corings with SAYD coefficients from \cite{BohmStef08,HassRang13}. Let $\bm{H}$ be a right $\times_A$-Hopf algebroid, $\bm{C}$ an $A$-coring as well as a right $\bm{H}$-module coring, \ie $\bm{C}$ is a right $\bm{H}$-module and for any $\bm{c} \in \bm{C}$ and $\bm{h} \in \bm{H}$,
\begin{itemize}
\item[(i)] $\ve_{\bm{C}}(\bm{c} \lt \bm{h}) = \ve_{\bm{C}}(\bm{c}) \lt \bm{h} =  \ve_{\bm{H}}(s(\ve_{\bm{C}}(\bm{c}))\bm{h})$,
\item[(ii)] $\D_{\bm{C}}(\bm{c} \lt \bm{h}) = \D_{\bm{C}}(\bm{c}) \lt \bm{h} = \bm{c}\ps{1}\lt \bm{h}\ps{1} \ot_A \bm{c}\ps{2}\lt \bm{h}\ps{2}$.
\end{itemize}
Let also $\bm{M}$ be a left-right SAYD module over $\bm{H}$. Then,
\begin{equation}\label{aux-cyclic-module-FY}
{\rm C}^{\bm{H}}_n(\bm{C}, \bm{M}) := \bm{C}^{\ot_A \, n+1} \ot_{\bm{H}} \bm{M}, \quad n \geq 0
\end{equation}
is a cyclic module by the operators
\begin{align}\label{cyclic-Y-face}
\begin{split}
 d_i:{\rm C}^{\bm{H}}_{n+1}(\bm{C}, \bm{M}) &\to {\rm C}^{\bm{H}}_n(\bm{C}, \bm{M}), \quad 0 \leq i \leq n+1 
 \\
 \left(\bm{c}^0 \ot_A \ldots \ot_A \bm{c}^{n+1}\right) \ot_{\bm{H}} \bm{m} &\mapsto  \left(\bm{c}^0 \ot_A \ldots \ot_A \ve(\bm{c}^i) \ot_A \ldots \ot_A \bm{c}^{n+1}\right) \ot_{\bm{H}} \bm{m},
\end{split}
\end{align}
\begin{align}\label{cyclic-Y-degeneracy}
\begin{split}
 s_i:{\rm C}^{\bm{H}}_{n-1}(\bm{C}, \bm{M}) &\to {\rm C}^{\bm{H}}_n(\bm{C}, \bm{M}), \quad 0 \leq i \leq n-1 \\
 \left(\bm{c}^0 \ot_A \ldots \ot_A \bm{c}^{n-1}\right) \ot_{\bm{H}} \bm{m}  &\mapsto \left(\bm{c}^0 \ot_A \ldots \ot_A \bm{c}^i\ps{1} \ot_A \bm{c}^i\ps{2} \ot_A \ldots \ot_A \bm{c}^{n-1}\right) \ot_{\bm{H}} \bm{m},
\end{split}
\end{align}
and
\begin{align}\label{cyclic-Y-cyclic}
 t_{n}:{\rm C}^{\bm{H}}_n(&\bm{C}, \bm{M}) \to  {\rm C}^{\bm{H}}_n(\bm{C}, \bm{M}), \\
\left(\bm{c}^0 \ot_A \ldots \ot_A \bm{c}^n\right) \ot_{\bm{H}} \bm{m} \mapsto &\left(\bm{c}^n \lt \bm{m}\ns{1} \ot_A \bm{c}^0 \ot_A \ldots \ot_A \bm{c}^{n-1}\right) \ot_{\bm{H}} \bm{m}\ns{0}.\nonumber
\end{align}

The cyclic (resp. periodic cyclic, negative cyclic, Hochschild) homology of the cyclic object \eqref{aux-cyclic-module-FY} is called the \emph{$\times_A$-Hopf-cyclic (resp. periodic cyclic, negative cyclic, Hochschild) homology} of the $\bm{H}$-module coring $\bm{C}$ with coefficients in the SAYD module $\bm{M}$ under the symmetry of the right $\times_A$-Hopf algebra $\bm{H}$, and it is denoted by ${\rm HC}^{\bm{H}}_{\bullet}(\bm{C}, \bm{M})$ (resp.  ${\rm HP}^{\bm{H}}_{\bullet}(\bm{C}, \bm{M})$,  ${\rm HN}^{\bm{H}}_{\bullet}(\bm{C}, \bm{M})$,  ${\rm HH}^{\bm{H}}_{\bullet}(\bm{C}, \bm{M})$).

If the algebra $A$ is the ground field $k$, then it is well known that the above notion of $\times_{A}$-Hopf algebroid reduces to the notion of Hopf algebra 
$\mathcal{H}$ where  invertibility of $\nu$ for 
$\bm{H}$   produces the antipode $S$ of 
$\mathcal{H}$ by the formula $
S(\mb{h})=\mb{h}^-\varepsilon(\mb{h}^+)$,
and the antipode provides the inverse to $\nu$ given by
$\nu^{-1}(\mb{h}\ot \mb{h}')=\mb{h}S(\mb{h}'\!\!\ps{1})\ot \mb{h}'\!\!\ps{2}$.
Then the above notion of $\times_A$-Hopf-cyclic object \eqref{aux-cyclic-module-FY} reduces to the notion of \emph{Hopf-cyclic object} for a Hopf algebra $\mathcal{H}$, an $\mathcal{H}$-module coalgebra $C$ and a SAYD module $M$ for 
$\mathcal{H}$ \cite{HajaKhalRangSomm04-II}. 
We  denote the corresponding cyclic homology  by ${\rm HC}^{\mathcal{H}}_{\bullet}(C, M)$ (resp.  ${\rm HP}^{\mathcal{H}}_{\bullet}(C, M)$,  ${\rm HN}^{\mathcal{H}}_{\bullet}(C, M)$,  ${\rm HH}^{\mathcal{H}}_{\bullet}(C, M)$). 

\subsection{Cyclic homology of algebra extensions}\label{Cyclic homology of algebra extensions}

We recall the relative cyclic homology of algebras from \cite{JaraStef06}. To this end, we first recall the cyclic tensor product from \cite{Quil89}. The cyclic tensor product of $B$-bimodules $M_1,\ldots,M_n$ is defined by
\begin{equation*}
M_1 \widehat{\ot}_B \ldots \widehat{\ot}_B M_n := (M_1\hspace{-0.175em} \ot_B \ldots \ot_B\hspace{-0.175em} M_n) \ot_{B^{\rm e}}\hspace{-0.175em} B.
\end{equation*}
As it is remarked in \cite{JaraStef06}, for any $B$-bimodule $M$ one has $M \ot_{B^e}B \cong M_B$, where $M_B = M/[M,B]$ and $[M,B]$ is the subspace of $M$ generated by the commutators, that is, the elements of the form $[m,b]:=m\cdot b-b\cdot m$.

The relative cyclic homology of an algebra extension $B \subseteq A$ is computed by the cyclic object ${\rm C}_{\bullet}(A\hspace{0.2em}|\hspace{0.1em}B)$, where 
\begin{equation}\label{aux-rel-cyclic-alg}
 {\rm C}_n(A\hspace{0.2em}|\hspace{0.1em}B) := A^{\widehat{\ot}_B \, n+1}\cong A^{\ot_B n+1}\ot_{B^{\rm e}}\hspace{-0.175em}B\cong  A^{\ot_B n+1}/[B,  A^{\ot_B n+1}] 
\end{equation}
 is a $B$-balanced cyclic tensor product, equipped with the morphisms
\begin{align}\label{aux-faces-Z-n(A/B, A)}
 d_i:{\rm C}_n(A\hspace{0.2em}|\hspace{0.1em}B) &\longrightarrow {\rm C}_{n-1}(A\hspace{0.2em}|\hspace{0.1em}B), \quad 0 \leq j \leq n, \\\notag
 a^0 \widehat{\ot}_B a^1 \widehat{\ot}_B \cdots \widehat{\ot}_B  a^n &\mapsto
 \left\{\begin{array}{cc}
         a^0 a^1 \widehat{\ot}_B \cdots \widehat{\ot}_B a^n, & i =0, \\
         a^0 \widehat{\ot}_B a^1 \widehat{\ot}_B \cdots \widehat{\ot}_B a^ia^{i+1} \widehat{\ot}_B \cdots \widehat{\ot}_B a^n, & 1 \leq i \leq n-1, \\
         a^na^0 \widehat{\ot}_B a^1 \widehat{\ot}_B \cdots \widehat{\ot}_B a^{n-1} & i = n,
       \end{array}
\right.
\end{align}
\begin{align}\label{aux-degeneracies-Z-n(A/B, A)}
 s_j:{\rm C}_n(A\hspace{0.2em}|\hspace{0.1em}B) &\longrightarrow {\rm C}_{n+1}(A\hspace{0.2em}|\hspace{0.1em}B), \quad 0
\leq j \leq n, \\\notag
a^0 \widehat{\ot}_B \cdots \widehat{\ot}_B a^n &\mapsto a^0 \widehat{\ot}_B a^1 \widehat{\ot}_B \cdots \widehat{\ot}_B a^j \widehat{\ot}_B 1 \widehat{\ot}_B a^{j+1} \widehat{\ot}_B \cdots \widehat{\ot}_B a^n,
\end{align}
and
\begin{align}\label{aux-cyclic-Z-n(A/B, A)}
 t_n:{\rm C}_n(A\hspace{0.2em}|\hspace{0.1em}B) &\longrightarrow {\rm C}_n(A\hspace{0.2em}|\hspace{0.1em}B),\\\notag
a^0 \widehat{\ot}_B \cdots \widehat{\ot}_B a^n &\mapsto a^n \widehat{\ot}_B a^0 \widehat{\ot}_B \cdots \widehat{\ot}_B a^{n-1}.
\end{align}
The cyclic (resp. periodic cyclic, negative cyclic, Hochschild) homology of the cyclic object \eqref{aux-rel-cyclic-alg} is called the \emph{relative cyclic (resp. periodic cyclic, negative cyclic, Hochschild) homology} of the extension $B \subseteq A$, and it is denoted by ${\rm HC}_\bullet (A\hspace{0.2em}|\hspace{0.1em}B)$ (resp.  ${\rm HP}_\bullet (A\hspace{0.2em}|\hspace{0.1em}B)$,  ${\rm HN}_\bullet (A\hspace{0.2em}|\hspace{0.1em}B)$,  ${\rm HH}_\bullet (A\hspace{0.2em}|\hspace{0.1em}B)$).

\subsection{Cyclic homology of corings associated to the extension}\label{Cyclic homology of corings associated to the extension}

The cyclic module \eqref{aux-rel-cyclic-alg} can also be interpreted in terms of the cyclic module \eqref{aux-cyclic-module-FY} of the Sweedler coring $A \ot_B\hspace{-0.175em} A$ of the extension with coefficients in the left-right SAYD module $A$ over the right $\times_A$-Hopf algebra $A^{\rm e}$, see Example \ref{example-SAYD-on-A-env}, via the isomorphism
\begin{align}\label{alg-ext-Z-to-aux-iso-coring-FY}
{\rm C}_n(A\hspace{0.2em}|\hspace{0.1em}B) & \lra {\rm C}^{A^{\rm e}}_n(A \ot_B\hspace{-0.175em} A, A),\\
a^0 \widehat{\ot}_B \cdots \widehat{\ot}_B a^n & \mapsto \left( (a^0 \ot_B 1) \ot_A \cdots \ot_A (a^n \ot_B 1)\right) \ot_{A^{\rm e}}  1. \nonumber
\end{align}
which can be easily checked to be a well defined  inverse to an obviously well defined map
\begin{align}\label{aux-iso-coring-FY-to-alg-ext-Z}
{\rm C}^{A^{\rm e}}_n(A \ot_B\hspace{-0.175em} A, A)  \lra &\ {\rm C}_n(A\hspace{0.2em}|\hspace{0.1em}B),\\
\left( (a^0 \ot_B \widetilde{a}^{0}) \ot_A \cdots \ot_A (a^n \ot_B \widetilde{a}^{n})\right) \ot_{A^{\rm e}}  1  & \mapsto \widetilde{a}^{n}a^0 \widehat{\ot}_B \widetilde{a}^{0}a^1 \widehat{\ot}_B \cdots \widehat{\ot}_B \widetilde{a}^{n-1}a^n.\nonumber 
\end{align}
The cyclic structure of ${\rm C}^{A^{\rm e}}_{\bullet}(A \ot_B\hspace{-0.175em} A, A)$ is then given explicitly by
\begin{align*}
 d_i:{\rm C}^{A^{\rm e}}_n(A \ot_B\hspace{-0.175em} A, A) &\lra {\rm C}^{A^{\rm e}}_{n-1}(A \ot_B\hspace{-0.175em} A, A),\quad 0\leq i \leq n, \\
 \left( (a^0 \ot_B 1) \ot_A \ldots \ot_A (a^n \ot_B 1)\right) \ot_{A^{\rm e}} 1 &\mapsto  \\
&\hspace{-4.5em}\left( (a^0 \ot_B 1) \ot_A \ldots \ot_A \ve(a^i \ot_B 1) \ot_A \ldots \ot_A (a^n \ot_B 1)\right) \ot_{A^{\rm e}} 1, \\
 s_j:{\rm C}^{A^{\rm e}}_n(A \ot_B\hspace{-0.175em} A, A) &\lra {\rm C}^{A^{\rm e}}_{n+1}(A \ot_B\hspace{-0.175em} A, A), \quad 0\leq j \leq n,\\
 \left( (a^0 \ot_B 1) \ot_A \ldots \ot_A (a^n \ot_B 1)\right) \ot_{A^{\rm e}} 1 &\mapsto \\
&\hspace{-4.5em}\left( (a^0 \ot_B 1) \ot_A \ldots \ot_A \D(a^j \ot_B 1) \ot_A \ldots \ot_A (a^n \ot_B 1)\right) \ot_{A^{\rm e}} 1, \\
 t_n:{\rm C}^{A^{\rm e}}_n(A \ot_B\hspace{-0.175em} A, A) &\lra {\rm C}^{A^{\rm e}}_n(A \ot_B\hspace{-0.175em} A, A), \\
 \left( (a^0 \ot_B 1) \ot_A \ldots \ot_A (a^n \ot_B 1)\right) \ot_{A^{\rm e}} 1 &\mapsto \\
& \hspace{-1.5cm}\left(  (a^n \ot_B 1) \ot_A (a^0 \ot_B 1) \ot_A \ldots \ot_A (a^{n-1} \ot_B 1)\right) \ot_{A^{\rm e}} 1.
\end{align*}

Given a $C$-Galois extension $B\subseteq A$, there is also an $A$-coring structure on $A\ot C$, called the Galois-coring of the extension, whose underlying $A$-bimodule structure reads as
\[
a'(a\ot c)=a'a\ot c,\ \ \ (a\ot c)a'=aa'\ns{0}\ot ca'\ns{1}.
\]
This can be understood as the coincidence of the \emph{canonical entwining}  
\[
c\ot a \mapsto can(can^{-1}(1\ot c)a)
\]
(of  the $C$-Galois extension $A$) with the \emph{Hopf-entwining} 
\[
c\ot a \mapsto a\ns{0}\ot c a\ns{1}
\]
for a right $\C{H}$-module coalgebra $C$ and a right $\C{H}$-comodule algebra $A$, \cite{BrzeHaja09,BrzeWisb-book}.

\subsection{The cyclic object for the Galois coring.}

The cyclic object structure of the cyclic object construction applied to the Galois coring is explicitly given by 
\begin{align}
 d_i:{\rm C}^{A^{\rm e}}_n(A \ot C, A) &\lra {\rm C}_{n-1}(A \ot C, A),\quad 0\leq i \leq n, \\
\left((a^0 \ot c^0) \ot_A \ldots \ot_A (a^n \ot c^n)\right) \ot_{A^{\rm e}} 1 &\mapsto\nonumber  \\
&\hspace{-2.5em}\left( (a^0 \ot c^0) \ot_A \ldots \ot_A \ve(a^i \ot c^i) \ot_A \ldots \ot_A (a^n \ot c^n)\right) \ot_{A^{\rm e}}  1,\nonumber
\end{align} 

\begin{align}
 s_j:{\rm C}^{A^{\rm e}}_n(A \ot C, A) &\lra {\rm C}^{A^{\rm e}}_{n+1}(A \ot C, A), \quad 0\leq j \leq n,\\
\left((a^0 \ot c^0) \ot_A \ldots \ot_A (a^n \ot c^n)\right) \ot_{A^{\rm e}} 1 &\mapsto\nonumber \\
&\hspace{-2.5em}\left( (a^0 \ot c^0) \ot_A \ldots \ot_A \D(a^j \ot c^j) \ot_A \ldots \ot_A (a^n \ot c^n)\right) \ot_{A^{\rm e}}  1,\nonumber  
\end{align}
 
\begin{align}
 t_n:{\rm C}^{A^{\rm e}}_n(A \ot C,A) &\lra {\rm C}^{A^{\rm e}}_n(A \ot C, A), \\
\left((a^0 \ot c^0) \ot_A \ldots \ot_A (a^n \ot c^n)\right) \ot_{A^{\rm e}}  1 &\mapsto \nonumber \\
&\hspace{-2.5em} \left( (a^n \ot c^n) \ot_A (a^0 \ot c^0) \ot_A \ldots \ot_A (a^{n-1} \ot c^{n-1})\right) \ot_{A^{\rm e}}  1.\nonumber
\end{align}

\subsection{The canonical map of cyclic objects}
Tensoring the canonical map \eqref{aux-canonical-map} of corings, we obtain a canonical map of cyclic objects
\begin{align}\label{cyclic-mod-A-tensor-C}
 {\rm C}^{A^{\rm e}}_n(A \ot_B\hspace{-0.175em} A, A) &\lra {\rm C}^{A^{\rm e}}_n(A \ot C, A), \\ 
 \left((a^0 \ot_B 1) \ot_A \ldots \ot_A (a^n \ot_B 1)\right) \ot_{A^{\rm e}} 1 &\mapsto \left((a^0 \ot e) \ot_A \ldots \ot_A (a^n \ot e)\right) \ot_{A^{\rm e}} 1.\nonumber
\end{align}

In the $C$-Galois case it is an isomorphism of cyclic objects with the inverse given by   
\begin{align}\label{cyclic-mod-A-tensor-C-inverse}
 {\rm C}^{A^{\rm e}}_n(A \ot C, A) &\lra {\rm C}^{A^{\rm e}}_n(A \ot_B\hspace{-0.175em} A, A), \\ 
\left((a^0 \ot c^0) \ot_A \ldots \ot_A (a^n \ot c^n)\right) \ot_{A^{\rm e}} 1 &\mapsto ((a^0c^0\nsb{1} \ot_B c^0\nsb{2}) \ot_A \ldots \ot_A (a^nc^n\nsb{1} \ot_B c^n\nsb{2})) \ot_{A^{\rm e}} 1,\nonumber
\end{align}
where $\tau:C \to A\ot_B\hspace{-0.175em} A$, $c\mapsto c\nsb{1} \ot_B c\nsb{2}$ is the translation map, \ie $\tau(c)= can^{-1}(1\ot c)$, for any $c\in C$. 

Note that both maps are uniquely determined by formulas \eqref{cyclic-mod-A-tensor-C} and \eqref{cyclic-mod-A-tensor-C-inverse} since the elements on which they are evaluated in \eqref{cyclic-mod-A-tensor-C}-\eqref{cyclic-mod-A-tensor-C-inverse} span their domains.

It identifies the relative cyclic homology of a $C$-Galois extension $B\subseteq A$ with the Hopf-cyclic homology of the Galois-coring $A\ot C$, with coefficients in the SAYD module $A$ over the $\times_A$-Hopf algebra $A^{\rm e}$.  

\section{Hopf-cyclic homology for module coalgebra-extensions}\label{Hopf-cyclic homology for  module coalgebra-Galois comodule algebra extensions}
\subsection{The SAYD module associated with a comodule algebra}\label{The SAYD module associated with a comodule algebra} 
\begin{theorem}[\bf{The inertial SAYD module}]\label{thm-SAYD}
For any right $\C{H}$-comodule algebra $A$ the quotient vector space
\begin{align}\label{M}
M:=(\C{H}\ot A)/\langle \mb{h}a'\ps{1}\ot aa'\ps{0}-\mb{h}\ot a'a \mid \mb{h}\in \C{H},\,a,a' \in A \rangle .
\end{align}
is a stable anti-Yetter--Drinfeld $H$-module with the left $H$-module structure 
\begin{align}\label{mod}
\begin{split}
\C{H}\ot M&\rightarrow M,\\
 \mb{h}'\ot [\mb{h}\ot a]&\mapsto \mb{h}'\cdot [\mb{h}\ot a]:=[\mb{h}'\mb{h}\ot a],
\end{split}
\end{align} 
and the right $\C{H}$-comodule structure given by
\begin{align}\label{comod}
\begin{split} 
M&\rightarrow M\ot \C{H} ,\\
 [\mb{h}\ot a]&\mapsto  [\mb{h}\ot a]\ns{0}\ot  [\mb{h}\ot a]\ns{1}:=[\mb{h}\ps{2} \ot a\ps{0}]\ot \mb{h}\ps{3} a\ps{1}S(\mb{h}\ps{1}).
\end{split}
\end{align}
\end{theorem}

\begin{proof} 
It is obvious that the left $\C{H}$-module structure on $\C{H}\ot A$ is well defined by the formula \eqref{mod}. The left module structure on its quotient $M$ is well defined since
\[
\mb{h}'(\mb{h}a'\ps{1}\ot aa'\ps{0}-\mb{h}\ot a'a)=(\mb{h}'\mb{h})a'\ps{1}\ot aa'\ps{0}-(\mb{h}'\mb{h})\ot a'a
\]
which means that elements  $\mb{h}a'\ps{1}\ot aa'\ps{0}-\mb{h}\ot a'a$ span a left $\C{H}$-submodule in the left $\C{H}$-module $\C{H}\ot A$.

As for the well-definedness of the right comodule map $M\rightarrow M\ot \C{H}$ as defined above, let us first compute
\begin{align*}
&(\mb{h}a'\ps{1}\ot aa'\ps{0})\ns{0}\ot (\mb{h}a'\ps{1}\ot aa'\ps{0})\ns{1}\\
&=((\mb{h}a'\ps{1})\ps{2}\ot (aa'\ps{0})\ps{0})\ot (\mb{h}a'\ps{1})\ps{3}(aa'\ps{0})\ps{1}S( (\mb{h}a'\ps{1})\ps{1})\\
&=(\mb{h}\ps{2}a'\ps{1}\ps{2}\ot a\ps{0}a'\ps{0}\ps{0})\ot \mb{h}\ps{3}a'\ps{1}\ps{3}a\ps{1}a'\ps{0}\ps{1}S( \mb{h}\ps{1}a'\ps{1}\ps{1})\\
&=(\mb{h}\ps{2}a'\ps{3}\ot a\ps{0}a'\ps{0})\ot \mb{h}\ps{3}a'\ps{4}a\ps{1}a'\ps{1}S(a'\ps{2}) S( \mb{h}\ps{1})\\
&=(\mb{h}\ps{2}a'\ps{2}\ot a\ps{0}a'\ps{0})\ot \mb{h}\ps{3}a'\ps{3}a\ps{1}\varepsilon( a'\ps{1}) S( \mb{h}\ps{1})\\
&=(\mb{h}\ps{2}a'\ps{1}\ot a\ps{0}a'\ps{0})\ot \mb{h}\ps{3}a'\ps{2}a\ps{1} S( \mb{h}\ps{1})\\
&=(\mb{h}\ps{2}a'\ps{0}\ps{1}\ot a\ps{0}a'\ps{0}\ps{0})\ot \mb{h}\ps{3}a'\ps{1}a\ps{1} S( \mb{h}\ps{1})\\
&=(\mb{h}\ps{2}a'\ps{0}\ps{1}\ot a\ps{0}a'\ps{0}\ps{0}-\mb{h}\ps{2}\ot a'\ps{0}a\ps{0} )\ot \mb{h}\ps{3}a'\ps{1}a\ps{1} S( \mb{h}\ps{1})\\
&\ \ \ \ +(\mb{h}\ps{2}\ot a'\ps{0}a\ps{0} )\ot \mb{h}\ps{3}a'\ps{1}a\ps{1} S( \mb{h}\ps{1})
\end{align*}
and compare it with
\begin{align*}
& (\mb{h}\ot a'a)\ns{0}\ot  (\mb{h}\ot a'a)\ns{1}\\
&=(\mb{h}\ps{2} \ot (a'a)\ps{0})\ot \mb{h}\ps{3}(a' a)\ps{1}S(\mb{h}\ps{1})\\
&=(\mb{h}\ps{2} \ot a'\ps{0}a\ps{0})\ot \mb{h}\ps{3}a'\ps{1} a\ps{1}S(\mb{h}\ps{1}).
\end{align*}
The difference is
\begin{align*}
&(\mb{h}a'\ps{1}\ot aa'\ps{0}-\mb{h}\ot a'a)\ns{0}\ot (\mb{h}a'\ps{1}\ot aa'\ps{0}-h\ot a'a)\ns{1}\\
&=(\mb{h}\ps{2}a'\ps{0}\ps{1}\ot a\ps{0}a'\ps{0}\ps{0}-\mb{h}\ps{2}\ot a'\ps{0}a\ps{0} )\ot \mb{h}\ps{3}a'\ps{1}a\ps{1} S( \mb{h}\ps{1})
\end{align*}
which means that elements  $\mb{h}a'\ps{1}\ot aa'\ps{0}-\mb{h}\ot a'a$ span a right $\C{H}$-subcomodule in the right $\C{H}$-comodule $\C{H}\ot A$.

The formula \eqref{comod} defines a right $\C{H}$-comodule structure on $\C{H}\ot A$ since
\begin{align*}
& [\mb{h}\ot a]\ns{0}\ns{0}\ot  [\mb{h}\ot a]\ns{0}\ns{1}\ot [\mb{h}\ot a]\ns{1}\\
&=[\mb{h}\ps{2} \ot a\ps{0}]\ns{0}\ot [\mb{h}\ps{2} \ot a\ps{0}]\ns{1} \ot \mb{h}\ps{3} a\ps{1}S(\mb{h}\ps{1})\\
&=[\mb{h}\ps{2}\ps{2} \ot a\ps{0}\ps{0}]\ot \mb{h}\ps{2}\ps{3}a\ps{0}\ps{1} S(\mb{h}\ps{2}\ps{1})\ot \mb{h}\ps{3} a\ps{1}S(\mb{h}\ps{1})\\
&=[\mb{h}\ps{3} \ot a\ps{0}]\ot \mb{h}\ps{4}\ps{3}a\ps{1} S(\mb{h}\ps{2})\ot \mb{h}\ps{5} a\ps{2}S(\mb{h}\ps{1})\\
&=[\mb{h}\ps{2} \ot a\ps{0}]\ot \mb{h}\ps{3}\ps{1}a\ps{1} S(\mb{h}\ps{1}\ps{2})\ot h\ps{3}\ps{2} a\ps{2}S(\mb{h}\ps{1}\ps{1})\\
&=[\mb{h}\ps{2} \ot a\ps{0}]\ot \mb{h}\ps{3}\ps{1}a\ps{1}\ps{1} S(\mb{h}\ps{1})\ps{1}\ot \mb{h}\ps{3}\ps{2} a\ps{1}\ps{2}S(\mb{h}\ps{1})\ps{2}\\
&=[\mb{h}\ps{2} \ot a\ps{0}]\ot (\mb{h}\ps{3}a\ps{1} S(\mb{h}\ps{1}))\ps{1}\ot  (\mb{h}\ps{3}a\ps{1} S(\mb{h}\ps{1}))\ps{2}\\
&= [\mb{h}\ot a]\ns{0}\ot  [\mb{h}\ot a]\ns{1}\ps{1}\ot [\mb{h}\ot a]\ns{1}\ps{2}
\end{align*}
and
\begin{align*}
& [\mb{h}\ot a]\ns{0}\varepsilon([\mb{h}\ot a]\ns{1})=[\mb{h}\ps{2} \ot a\ps{0}]\varepsilon(\mb{h}\ps{3} a\ps{1}S(\mb{h}\ps{1}))\\
&=[\mb{h}\ps{2} \ot a\ps{0}]\varepsilon(\mb{h}\ps{3}) \varepsilon(a\ps{1})\varepsilon(S(\mb{h}\ps{1}))=\varepsilon(\mb{h}\ps{1})\mb{h}\ps{2}\varepsilon(\mb{h}\ps{3}) \ot a\ps{0} \varepsilon(a\ps{1})\\
&=[\mb{h}\ot a].
\end{align*}

Now we check the stability  condition. We have,
\begin{align*}
& [\mb{h}\ot a]\ns{1} [\mb{h}\ot a]\ns{0} = \mb{h}\ps{3} a\ps{1}S(\mb{h}\ps{1})[\mb{h}\ps{2} \ot a\ps{0}]\\
&= [\mb{h}\ps{3} a\ps{1}S(\mb{h}\ps{1})\mb{h}\ps{2} \ot a\ps{0}]=[\varepsilon(\mb{h}\ps{1}) \mb{h}\ps{2} a\ps{1} \ot a\ps{0}] \\
&= [\mb{h} a\ps{1} \ot a\ps{0}]= [\mb{h} a\ps{1} \ot 1\cdot a\ps{0} -\mb{h}  \ot a\cdot 1] + [\mb{h}\ot a]=[\mb{h}\ot a],
\end{align*}
as such, we see that the stability holds only in the level of the quotient $M$. Note that it does not hold on the level of $\C{H}\ot A$.

Finally, we check the AYD-condition. We have,
\begin{align*}
& (\mb{h}'\cdot [\mb{h}\ot a])\ns{0}\ot (\mb{h}' \cdot [\mb{h}\ot a])\ns{1}\\
&=[\mb{h}'\mb{h}\ot a]\ns{0}\ot [\mb{h}' \mb{h}\ot a]\ns{1}\\
&=[(\mb{h}'\mb{h})\ps{2} \ot a\ps{0}]\ot (\mb{h}'\mb{h})\ps{3} a\ps{1}S((\mb{h}'\mb{h})\ps{1})\\
&=[\mb{h}'\ps{2}\mb{h}\ps{2} \ot a\ps{0}]\ot \mb{h}'\ps{3}\mb{h}\ps{3} a\ps{1}S(\mb{h}'\ps{1}\mb{h}\ps{1})\\
&=[\mb{h}'\ps{2}\mb{h}\ps{2} \ot a\ps{0}]\ot \mb{h}'\ps{3}\mb{h}\ps{3} a\ps{1}S(\mb{h}\ps{1})S(\mb{h}'\ps{1})\\
&=\mb{h}'\ps{2} \cdot [\mb{h}\ps{2} \ot a\ps{0}]\ot \mb{h}'\ps{3}(\mb{h}\ps{3} a\ps{1}S(\mb{h}\ps{1}))S(\mb{h}'\ps{1})\\
&=\mb{h}'\ps{2}\cdot [\mb{h}\ot a]\ns{0}\ot \mb{h}'\ps{3}[\mb{h}\ot a]\ns{1}S(\mb{h}'\ps{1}).
\end{align*} 
\end{proof}
\noindent{\bf Remark.} Note that provided $\C{H}\ot A$ is equipped with the following $A$-bimodule structure
$$a'(\mb{h}\ot a):=\mb{h}\ot a'a,\ \ \ \ (\mb{h}\ot a)a':=\mb{h}a'_{(1)}\ot aa'_{(0)},$$
$M$ can be written equivalently as 
\begin{align}\label{M}
M=(\C{H}\ot A)/[A, \C{H}\ot A]=(\C{H}\ot A)\ot_{A^{\rm e}}A.
\end{align}
\subsection{The Hopf algebra change in inertial Hopf-cyclic homology}\label{?} In the classical case the inertial cyclic space depends only on the group $H$ and the choice of the group homomorphism $H\rightarrow G$ is irrelevant. In the noncommutative case, when $G$ corresponds to a Hopf algebra $\C{H}$ and the group homomorphism $H\rightarrow G$ to the structure of a right $\C{H}$-module coalgebra $C$, $C$ can be not a Hopf algebra, and the Hopf-cyclic object for $C$ doesn't make sense. However, the choice of a Hopf algebra  $\C{H}$ is again irrelevant, as far as another choice is appropriatly compatible.  The latter means that we  allow another Hopf algebra $\widetilde{\C{H}}$ coaction  on $A$ and a Hopf algebra map $\varphi: \widetilde{\C{H}}\rightarrow \C{H}$ making the following diagram commute.
\begin{equation}\label{diagram}
\begin{gathered}
\xymatrix@C=1.5cm{
                                 &A\ot \widetilde{\C{H}}\ar[d]^{A\otimes \varphi}\\
A\ar[r]^{\alpha}\ar[ur]^{\widetilde{\alpha}} &A\ot \C{H}.}
\end{gathered}
\end{equation}
To prove that irrelevance, we will use the three following lemmas. The first one is stated as Example 4.4 (c) in \cite{JaraStef06} without a proof which we provide for the sake of completeness.
 
\begin{lemma}[Hopf algebra change for SAYD modules]\label{bc-SAYD}
For any left-right SAYD module $\widetilde{M}$ over a Hopf algebra  $\widetilde{\C{H}}$ any Hopf algebra map  $\varphi: \widetilde{\C{H}}\rightarrow \C{H}$ induces  a structure of a left-right SAYD module over $\C{H}$ on    
$\C{H}\ot_{\widetilde{\C{H}}}\hspace{-0.175em}\widetilde{M}$.
\end{lemma}
\begin{proof}
The left $\C{H}$-module structure of $\C{H}\ot_{\widetilde{\C{H}}}\widetilde{M}$ is obvious
\begin{align}\label{bc-SAYD-mult}
\mb{h}'(\mb{h}\ot_{\widetilde{\C{H}}}\widetilde{m}):=\mb{h}'\mb{h}\ot_{\widetilde{\C{H}}}\widetilde{m}.
\end{align}
The right $\C{H}$-comodule structure of $\C{H}\ot_{\widetilde{\C{H}}}\widetilde{M}$ is defined as follows
\begin{align}\label{bc-SAYD-comult}
(\mb{h}\ot_{\widetilde{\C{H}}}\widetilde{m})_{(0)}\ot (\mb{h}\ot_{\widetilde{H}}\widetilde{m})_{(1)}:=(\mb{h}_{(2)}\ot_{\widetilde{\C{H}}}\widetilde{m}_{(0)})\ot \mb{h}_{(3)}\varphi(\widetilde{m}_{(1)})S(\mb{h}_{(1)}).
\end{align}
The counitality of $\C{H}\ot_{\widetilde{\C{H}}}\hspace{-0.175em}\widetilde{M}$ follows easily:
\begin{equation}
\begin{gathered}
(\mb{h}\ot_{\widetilde{\C{H}}}\widetilde{m}))_{(0)}\varepsilon((\mb{h}\ot_{\widetilde{\C{H}}}\widetilde{m}))_{(1)})= (\mb{h}_{(2)}\ot_{\widetilde{H}}\widetilde{m}_{(0)})\varepsilon( \mb{h}_{(3)}\varphi(\widetilde{m}_{(1)})S(\mb{h}_{(1)}))\\
= (\mb{h}_{(2)}\ot_{\widetilde{\C{H}}}\widetilde{m}_{(0)})\varepsilon( \mb{h}_{(3)})\varepsilon(\varphi(\widetilde{m}_{(1)}))\varepsilon(S(\mb{h}_{(1)})))
= (\mb{h}_{(2)}\ot_{\widetilde{\C{H}}}\widetilde{m}_{(0)})\varepsilon( \mb{h}_{(3)})\varepsilon(\widetilde{m}_{(1)})\varepsilon(\mb{h}_{(1)})\\
=  \varepsilon(\mb{h}_{(1)})\mb{h}_{(2)}\varepsilon( \mb{h}_{(3)})\ot_{\widetilde{\C{H}}}\widetilde{m}_{(0)}\varepsilon(\widetilde{m}_{(1)})=\mb{h}\ot_{\widetilde{\C{H}}}\widetilde{m}.
\end{gathered}
\end{equation}
The coassociativity of $\C{H}\ot_{\widetilde{\C{H}}}\hspace{-0.175em}\widetilde{M}$ is slightly more involved.
\begin{equation}
\begin{gathered}
(\mb{h}\ot_{\widetilde{\C{H}}}\widetilde{m}))_{(0)}\ot (\mb{h}\ot_{\widetilde{\C{H}}}\widetilde{m}))_{(1)(1)}\ot (\mb{h}\ot_{\widetilde{\C{H}}}\widetilde{m}))_{(1)(2)}\\
=(\mb{h}_{(2)}\ot_{\widetilde{\C{H}}}\widetilde{m}_{(0)})\ot 
 (\mb{h}_{(3)}\varphi(\widetilde{m}_{(1)})S(\mb{h}_{(1)}))_{(1)}\ot  (\mb{h}_{(3)}\varphi(\widetilde{m}_{(1)})S(\mb{h}_{(1)}))_{(2)}\\
=(\bm{h}_{(2)}\ot_{\widetilde{\C{H}}}\widetilde{m}_{(0)})\ot 
\mb{h}_{(3)(1)}\varphi(\widetilde{m}_{(1)})_{(1)}S(\mb{h}_{(1)})_{(1)}\ot  h_{(3)(2)
}\varphi(\widetilde{m}_{(1)})_{(2)}S(\mb{h}_{(1)})_{(2)}\\
=(\bm{h}_{(2)}\ot_{\widetilde{\C{H}}}\widetilde{m}_{(0)})\ot 
\mb{h}_{(3)(1)}\varphi(\widetilde{m}_{(1)(1)})S(\mb{h}_{(1)(2)})\ot  \mb{h}_{(3)(2)
}\varphi(\widetilde{m}_{(1)(2)})S(\mb{h}_{(1)(1)})\\
=(\bm{h}_{(2)}\ot_{\widetilde{\C{H}}}\widetilde{m}_{(0)})\ot 
\mb{h}_{(3)(1)}\varphi(\widetilde{m}_{(1)(1)})S(\mb{h}_{(1)(2)})\ot  \mb{h}_{(3)(2)
}\varphi(\widetilde{m}_{(1)(2)})S(\mb{h}_{(1)(1)})\\
=(\bm{h}_{(3)}\ot_{\widetilde{\C{H}}}\widetilde{m}_{(0)})\ot 
\mb{h}_{(4)}\varphi(\widetilde{m}_{(1)})S(\mb{h}_{(2)})\ot  \mb{h}_{(5)
}\varphi(\widetilde{m}_{(2)})S(\mb{h}_{(1)})\\ 
=(\bm{h}_{(2)(2)}\ot_{\widetilde{\C{H}}}\widetilde{m}_{(0)(0)})\ot 
\mb{h}_{(2)(3)}\varphi(\widetilde{m}_{(0)(1)})S(\mb{h}_{(2)(1)})\ot  \mb{h}_{(3)
}\varphi(\widetilde{m}_{(1)})_{(2)}S(\mb{h}_{(1)})\\
=(\mb{h}_{(2)}\ot_{\widetilde{\C{H}}}\widetilde{m}_{(0)})_{(0)}\ot 
(\mb{h}_{(2)}\varphi(\widetilde{m}_{(0)})_{(1)}\ot  \mb{h}_{(3)
}\varphi(\widetilde{m}_{(1)})S(\mb{h}_{(1)})\\
=(\mb{h}\ot_{\widetilde{\C{H}}}\widetilde{m})_{(0)(0)}\ot 
(\mb{h}\ot_{\widetilde{\C{H}}}\widetilde{m})_{(0)(1)}\ot  (\mb{h}\ot_{\widetilde{\C{H}}}\widetilde{m})_{(1)}.
\end{gathered}
\end{equation}
The stability of $\C{H}\ot_{\widetilde{\C{H}}}\hspace{-0.175em}\widetilde{M}$ follows from the  computation using (\ref{bc-SAYD-mult}),  (\ref{bc-SAYD-comult}) and the stability of $\widetilde{M}$.
\begin{equation}
\begin{gathered}
\ (\mb{h}\ot_{\widetilde{\C{H}}}\widetilde{m})_{(1)}(\mb{h}\ot_{\widetilde{\C{H}}}\widetilde{m})_{(0)}=\mb{h}_{(3)}\varphi(\widetilde{m}_{(1)})S(\mb{h}_{(1)})\mb{h}_{(2)}\ot_{\widetilde{\C{H}}}\widetilde{m}_{(0)}\\
 = \  \mb{h}_{(2)}\varphi(\widetilde{m}_{(1)})\varepsilon(\mb{h}_{(1)})\ot_{\widetilde{\C{H}}}\widetilde{m}_{(0)}=\varepsilon(\mb{h}_{(1)})\mb{h}_{(2)}\varphi(\widetilde{m}_{(1)})\ot_{\widetilde{\C{H}}}\widetilde{m}_{(0)}\\
 = \   \mb{h}\varphi(\widetilde{m}_{(1)})\ot_{\widetilde{\C{H}}}\widetilde{m}_{(0)}= \mb{h}\ot_{\widetilde{\C{H}}}\widetilde{m}_{(1)}\widetilde{m}_{(0)}=\mb{h}\ot_{\widetilde{\C{H}}}\widetilde{m}.
\end{gathered}
\end{equation}
The AYD property of $\C{H}\ot_{\widetilde{\C{H}}}\hspace{-0.175em}\widetilde{M}$ follows from the computation using (\ref{bc-SAYD-mult}),  (\ref{bc-SAYD-comult}) and the  AYD property of $\widetilde{M}$.
\begin{equation}
\begin{gathered}
(\mb{h}'(\mb{h}\ot_{\widetilde{\C{H}}}\widetilde{m}))_{(0)}\ot (\mb{h}'(\mb{h}\ot_{\widetilde{\C{H}}}\widetilde{m}))_{(1)}=(\mb{h}'\mb{h}\ot_{\widetilde{\C{H}}}\widetilde{m})_{(0)}\ot (\mb{h}'\mb{h}\ot_{\widetilde{\C{H}}}\widetilde{m})_{(1)}\\
=((\mb{h}'\mb{h})_{(2)}\ot_{\widetilde{\C{H}}}\widetilde{m}_{(0)})\ot (\mb{h}'\mb{h})_{(3)}\varphi(\widetilde{m}_{(1)})S((\mb{h}'\mb{h})_{(1)})\\
=(\mb{h}'_{(2)}\mb{h}_{(2)}\ot_{\widetilde{\C{H}}}\widetilde{m}_{(0)})\ot \mb{h}'_{(3)}\mb{h}_{(3)}\varphi(\widetilde{m}_{(1)})S(\mb{h}'_{(1)}\mb{h}_{(1)})\\
=\mb{h}'_{(2)}(\mb{h}_{(2)}\ot_{\widetilde{\C{H}}}\widetilde{m}_{(0)})\ot \mb{h}'_{(3)}(\mb{h}_{(3)}\varphi(\widetilde{m}_{(1)})S(\mb{h}_{(1)}))S(\mb{h}'_{(1)})\\
=\mb{h}'_{(2)}(\mb{h}\ot_{\widetilde{\C{H}}}\widetilde{m})_{(0)}\ot \mb{h}'_{(3)}(\mb{h}\ot_{\widetilde{\C{H}}}\widetilde{m})_{(1)}S(\mb{h}'_{(1)}).
\end{gathered}
\end{equation}
\end{proof}
\begin{lemma}[Hopf-algebra-change-invariance of  Hopf-cyclic objects]\label{bc-HC-obj} For any right $\C{H}$-module coalgebra $C$ 
the collection of maps 
\begin{equation}\label{bc-Hopf-cyclic}
\begin{gathered}
{\rm C}^{\widetilde{\C{H}}}_n(C, \widetilde{M})\rightarrow {\rm C}^{\C{H}}_n(C, \C{H}\ot_{\widetilde{\C{H}}}\hspace{-0.175em}\widetilde{M})
,\\
(c^0\ot\cdots\ot c^n)\ot_{\widetilde{\C{H}}}\widetilde{m}\mapsto (c^0\ot\cdots\ot c^n)\ot_{\C{H}}(1\ot_{\widetilde{H}}\widetilde{m})
\end{gathered}
\end{equation}
form an isomorphism of cyclic objects.
\end{lemma}
\begin{proof} Bijectivity of the above maps is obvious. 
What we have to prove is that the above maps commute with faces and degeneracies induced by the counit and the comultiplication of the coalgebra $C$, and the cyclic operator coming from the comodule  structures of $\widetilde{M}$ and $\C{H}\ot_{\widetilde{\C{H}}}\hspace{-0.175em}\widetilde{M}$. The fact of commuting with faces and codegeneracies follows obviously from the form (\ref{bc-Hopf-cyclic}) of the maps under consideration. Commutation with the cyclic operator reads as the identity
\begin{equation}\label{?}
\begin{gathered}
(c^n(1\ot_{\widetilde{\C{H}}}\widetilde{m})\ns{1}\ot c^0\cdots\ot c^{n-1})\ot_{\C{H}}(1\ot_{\widetilde{\C{H}}}\widetilde{m})\ns{0}\\
=  (c^n\varphi(\widetilde{m}\ns{1})\ot c^0\cdots\ot c^{n-1})\ot_{\C{H}}(1\ot_{\widetilde{\C{H}}}\widetilde{m}\ns{0})
\end{gathered}
\end{equation}
following immediately from the identity
 \begin{equation}\label{?}
\begin{gathered}
(1\ot_{\widetilde{\C{H}}}\widetilde{m})\ns{0}\ot (1\ot_{\widetilde{\C{H}}}\widetilde{m})\ns{1}
=  (1\ot_{\widetilde{\C{H}}}\widetilde{m}\ns{0})\ot  \varphi(\widetilde{m}\ns{1})
\end{gathered}
\end{equation}
which is a special case of (\ref{bc-SAYD-comult}) for $\mb{h}=1$.
\end{proof}
\begin{lemma}\label{nc-bry-iso}
In the situation of the commutative diagram (\ref{diagram}) and for $M$ (resp. $\widetilde{M}$) being an inertial SAYD module over $\C{H}$ (resp. $\widetilde{\C{H}}$) defined as in Theorem \ref{thm-SAYD} there is an isomorphism of  SAYD modules over $\C{H}$
$$\C{H}\ot_{\widetilde{\C{H}}}\hspace{-0.175em}\widetilde{M}\rightarrow M.$$
\end{lemma}
\begin{proof} For  SAYD modules $M$ over $\C{H}$  and $\widetilde{M}$ over 
$\widetilde{\C{H}}$ presented as 
$$M=(\C{H}\ot A)/[A, \C{H}\ot A],\ \ \ \widetilde{M}=(\widetilde{\C{H}}\ot A)/[A, \widetilde{\C{H}}\ot A]$$
 we define two maps:
\begin{equation}\label{alpha}
\begin{gathered}
\alpha: \C{H}\ot_{\widetilde{\C{H}}}\widetilde{M}\rightarrow M,\ \ \ \alpha(\mb{h}\ot_{\widetilde{\C{H}}}[\widetilde{\mb{h}}\ot a]):= [\mb{h}\varphi(\widetilde{\mb{h}})\ot a],\\
\omega: M\rightarrow \C{H}\ot_{\widetilde{\C{H}}}\widetilde{M},\ \ \ \ \omega([\mb{h}\ot a]):=\mb{h}\ot_{\widetilde{\C{H}}}[1\ot a].
\end{gathered}
\end{equation}
They both are well defined by the second description of $M$ as $(\C{H}\ot A)\ot_{A^{\rm e}}A$ and the fact that $\C{H}\ot A$ is an $(\C{H}, A^{\rm e})$-bimodule. They are inverse one to each other by balancedness of the tensor product over $\widetilde{\C{H}}$. It is obvious that $\alpha$ is a left $\C{H}$-module map. 

While proving that $\alpha$ is  also a right $\C{H}$-comodule map we can use surjectivity of $\omega$ to simplify the calculation by restricting it to   elements of the form $\mb{h}\ot_{\widetilde{\C{H}}}[1\ot a]$. Then we start from the identities (\ref{bc-SAYD-comult}) and next (\ref{comod}) to get 
\begin{equation}\label{comult-red}
\begin{gathered}
(\mb{h}\ot_{\widetilde{\C{H}}}[1\ot a])_{(0)}\ot (\mb{h}\ot_{\widetilde{\C{H}}}[1\ot a])_{(1)}=(\mb{h}_{(2)}\ot_{\widetilde{\C{H}}}[1\ot a]_{(0)})\ot \mb{h}_{(3)}\varphi([1\ot a]_{(1)})S(\mb{h}_{(1)})\\
=(\mb{h}_{(2)}\ot_{\widetilde{\C{H}}}[1\ot a_{(0)}])\ot \mb{h}_{(3)}\varphi(a_{(1)})S(\mb{h}_{(1)}).
\end{gathered}
\end{equation}
Finally, applying $\alpha\ot \C{H}$ to  (\ref{comult-red}) and using  (\ref{alpha})               we prove right $\C{H}$-colinearity of $\alpha$ as follows
\begin{equation}
\begin{gathered}
\alpha((\mb{h}\ot_{\widetilde{\C{H}}}[1\ot a])_{(0)})\ot (\mb{h}\ot_{\widetilde{\C{H}}}[1\ot a])_{(1)}=[\mb{h}_{(2)}\ot_{\widetilde{\C{H}}}a_{(0)}]\ot \mb{h}_{(3)}\varphi (a_{(1)})S(\mb{h}_{(1)})\\
=[\mb{h}\ot_{\widetilde{\C{H}}}a]_{(0)}\ot [\mb{h}\ot_{\widetilde{\C{H}}}a]_{(1)}=\alpha(\mb{h}\ot_{\widetilde{\C{H}}}[1\ot a])_{(0)}\ot \alpha(\mb{h}\ot_{\widetilde{\C{H}}}[1\ot a])_{(1)}.
\end{gathered}
\end{equation} 
\end{proof} 
The following corollary of  Lemma \ref{bc-SAYD},   Lemma \ref{bc-HC-obj} and Lemma \ref{nc-bry-iso}   is the main result of the present paragraph.
\begin{proposition}[Hopf-algebra-change-invariance of the inertial Hopf-cyclic object]\label{basechange}\hspace{1em} Under the assumptions of the Lemma \ref{nc-bry-iso}, for any right $\C{H}$-module coalgebra $C$, the collection of maps
\begin{equation}\label{Hopf-alg-change}
\begin{gathered}
{\rm C}^{\widetilde{\C{H}}}_n(C, \widetilde{M}) \rightarrow {\rm C}^{\C{H}}_n(C, M),\\
\left(c^0\ot\cdots\ot c^n\right)\ot_{\widetilde{\C{H}}}(1\ot a)\mapsto (c^0\ot\cdots\ot c^n)\ot_{\C{H}}(1\ot a)
\end{gathered}
\end{equation}
forms an isomorphism of cyclic objects.
\end{proposition}
\begin{proof}
We can decompose (\ref{Hopf-alg-change}) as follows.
\begin{align}\label{Hopf-alg-change-decomp}
{\rm C}^{\widetilde{\C{H}}}_n(C, \widetilde{M})&\rightarrow {\rm C}^{\C{H}}_n(C, \C{H}\ot_{\widetilde{\C{H}}}\hspace{-0.175em}\widetilde{M})
\rightarrow {\rm C}^{\C{H}}_n(C, M),\\
(c^0\ot\cdots\ot c^n)\ot_{\widetilde{\C{H}}}\hspace{-0.27em}(1\ot a)\mapsto  (c^0\ot & \cdots\ot c^n)\ot_{\C{H}}\hspace{-0.27em}\left(1\ot_{\widetilde{\C{H}}}(1\ot a)\right)
\mapsto (c^0\ot\cdots\ot c^n)\ot_{\C{H}}\hspace{-0.27em}(1\ot a).\nonumber
\end{align}
Since by Lemma \ref{bc-HC-obj}  the left hand side arrows form an isomorphism of cyclic objects, and the  right hand side arrows  do so by Lemma \ref{nc-bry-iso}, the composites form an isomorphism of cyclic objects as well. 
\end{proof} 
\subsection{The isomorphism of the coring and coalgebra Hopf-cyclic objects}\label{The isomorphism of the associated cyclic objects} Let  us assume that there is given a right $\C{H}$-module coalgebra $C$ action on a right $\C{H}$-comodule algebra $A$.

Now we show a canonical isomorphism between the Hopf-cyclic object for the\linebreak $C$-Galois coring and the Hopf-cyclic homology of the coalgebra $C$ with coefficients in the SAYD module of Theorem \ref{thm-SAYD}. To this end, let us first note the following lemma.

\begin{lemma}\label{A-right-to-left-balanced} 
Given a right module coalgebra $C$ coaction on a right comodule algebra $A$ as above, we have
\begin{align*}
& ((1 \ot c^0) \ot_A \cdots \ot_A (1  \ot c^n))\ot_{A^{\rm e}}a'a\\
& =((1 \ot c^0a'\ps{1}) \ot_A  (1  \ot c^{1}a'\ps{2})\ot\cdots \ot_A (1  \ot c^n a'\ps{n+1}))\ot_{A^{\rm e}}aa'\ps{0}.
\end{align*}
for any $a,a' \in A$, and any $c^0,\ldots,c^n \in C$.
\end{lemma}

\begin{proof}
The  claim follows at once from the observation
\begin{align*}
& ((1 \ot c^0) \ot_A \cdots \ot_A (1  \ot c^n))\ot_{A^{\rm e}}a'a\\
&= ((1 \ot c^0) \ot_A \cdots \ot_A (1  \ot c^n))a'\ot_{A^{\rm e}}a\\
& =a'\ps{0}((1 \ot c^0a'\ps{1}) \ot_A  (1  \ot c^{1}a'\ps{2})\ot\cdots \ot_A (1  \ot c^n a'\ps{n+1}))\ot_{A^{\rm e}}a\\
& =((1 \ot c^0a'\ps{1}) \ot_A  (1  \ot c^{1}a'\ps{2})\ot\cdots \ot_A (1  \ot c^n a'\ps{n+1}))\ot_{A^{\rm e}}aa'\ps{0}.
\end{align*}
\end{proof}
We are now ready to prove our main technical result.
\begin{proposition}\label{thm-A-C-to-C-H-M}
For any Hopf algebra  $\C{H}$ and a right $\C{H}$-module coalgebra $C$ coaction on a right $\C{H}$-comodule algebra~$A$ the map 
\begin{align}\label{A-C-to-C-H-M}
\begin{split}
&\Phi: {\rm C}^{A^{\rm e}}_{\bullet}(A \ot C, A)\lra {\rm C}^{\mathcal{H}}_{\bullet}(C, M)\\
& \Big((a^0 \ot c^0) \ot_A \ldots \ot_A (a^n  \ot c^n)\Big) \ot_{A^{\rm e}} 1
\mapsto \Big(c^0 a^{1}\ps{1}\ldots a^{n}\ps{1} \ot  c^1 a^{2}\ps{2}\ldots a^{n}\ps{2} \ot\ldots\\
&\ \ \ \ \ \ldots \ot c^{n-2} a^{n-1}\ps{n-1}a^n\ps{n-1}  \ot c^{n-1} a^n\ps{n}\ot c^n\Big) \ot_{\C{H}} [1\ot a^0a^{1}\ps{0}\ldots a^{n}\ps{0}]
\end{split}
\end{align}
is a well defined morphism of cyclic objects. 
\end{proposition}
\begin{proof} 
Let us consider first an auxiliary mapping
\begin{align}\label{Phi-tilde}
\begin{split}
&\widetilde{\Phi}: (A\ot C)^{\ot \bullet+1}\ot A \lra {\rm C}^{\mathcal{H}}_{\bullet}(C, M)\\
& ((a^0 \ot c^0) \odots (a^n  \ot c^n)) \ot a' \mapsto \\
&(c^0 a^{1}\ps{1}\ldots a^{n}\ps{1} \ot  c^1 a^{2}\ps{2}\ldots a^{n}\ps{2} \odots c^{n-1} a^n\ps{n}\ot c^n) \ot_{\mathcal{H}} [1\ot a'a^0a^{1}\ps{0}\ldots a^{n}\ps{0}].
\end{split}
\end{align}
We note that for $1\leq k \leq n$,
\begin{align*}
&\widetilde{\Phi}\Big(((a^0 \ot c^0) \odots (a^{k-1}  \ot c^{k-1}) \ot (a^k  \ot c^k)\odots (a^n  \ot c^n)) \ot a'\Big)   \\
&= \widetilde{\Phi}\Big(((a^0 \ot c^0) \odots (a^{k-1}a^k\ps{0}  \ot c^{k-1} a^k\ps{1}) \ot (1  \ot c^k)\odots (a^n  \ot c^n)) \ot a'\Big),
\end{align*}
what proves that $\widetilde{\Phi}$ factors through the tensor product of copies of $A\ot C$ balanced 
over~$A$, and that by Lemma \ref{A-right-to-left-balanced}
\begin{align*}
& \widetilde{\Phi}\Big(((a'a^0 \ot c^0)\odots (a^n  \ot c^n)) \ot 1\Big)  \\
&=  (c^0 a^{1}\ps{1}\ldots a^{n}\ps{1} \ot  c^1 a^{2}\ps{2}\ldots a^{n}\ps{2} \odots c^{n-1} a^n\ps{n}\ot c^n) \ot_{\mathcal{H}} [1\ot a'a^0a^{1}\ps{0}\ldots a^{n}\ps{0}]  \\
&=  (c^0 a^{1}\ps{1}\ldots a^{n}\ps{1} \ot  c^1 a^{2}\ps{2}\ldots a^{n}\ps{2} \ot \ldots \\
& \hspace{3cm}\ldots \ot c^{n-1} a^n\ps{n}\ot c^n) \ot_{\mathcal{H}} [a'\ps{1}\ot a^0a^{1}\ps{0}\ldots a^{n}\ps{0}a'\ps{0}] \\
&=  (c^0 a^{1}\ps{1}\ldots a^{n}\ps{1}a'\ps{1} \ot  c^1 a^{2}\ps{2}\ldots a^{n}\ps{2}a'\ps{2} \ot \ldots \\
& \hspace{3cm}\ldots \ot c^{n-1} a^n\ps{n}a'\ps{n}\ot c^n a'\ps{n+1}) \ot_{\mathcal{H}} [1\ot a^0a^{1}\ps{0}\ldots a^{n}\ps{0}a'\ps{0}]  \\
&=  \widetilde{\Phi}\Big(((a^0 \ot c^0)\odots (a^na'\ps{0}  \ot c^n a'\ps{1})) \ot 1\Big),
\end{align*}
what proves that  $\eqref{Phi-tilde}$ factors through the tensor product with $A$ balanced over $A^{\rm e}$. 
As a result, \eqref{Phi-tilde} induces a well defined map \eqref{A-C-to-C-H-M}.

Let us now show that the mapping \eqref{A-C-to-C-H-M} commutes with the cyclic structure maps. To this end, we start with the commutation with the faces. Namely, for $0 \leq i \leq n$, we have
\begin{align*}
& \Phi(\d_i((a^0 \ot c^0) \ot_A \ldots \ot_A (a^n \ot c^n) \ot_{A^{\rm e}} \ot 1))  \\
&=  \Phi((a^0 \ot c^0) \ot_A \ldots \ot_A \ve(a^i \ot c^i) \ot_A \ldots \ot_A (a^n \ot c^n) \ot_{A^{\rm e}} \ot 1)  \\
&=  \ve(c^i)\Phi((a^0 \ot c^0) \ot_A \ldots \ot_A (a^{i-1} \ot c^{i-1}) \ot_A (a^ia^{i+1} \ot c^{i+1}) \ot_A (a^n \ot c^n) \ot_{A^{\rm e}} \ot 1)  \\
&=  \ve(c^i) \Big(c^0\cdot a^{1}\ps{1}\ldots a^{n}\ps{1} \ot  c^1\cdot a^{2}\ps{2}\ldots a^{n}\ps{2} \ot \ldots \\
& \ldots \ot c^{i-1}\cdot a^{i}\ps{i}a^{i+1}\ps{i}\ldots  a^{n}\ps{i} \ot c^{i+1}\cdot a^{i+2}\ps{i+1}\ldots  a^{n}\ps{i+1} \ot\ldots \\
& \ldots \ot c^{n-1}\cdot a^n\ps{n-1}\ot c^n\Big) \ot_{\C{H}} [1\ot a^0a^{1}\ps{0}\ldots a^{n}\ps{0}]  \\
&=  \Big(c^0\cdot a^{1}\ps{1}\ldots a^{n}\ps{1} \ot  c^1\cdot a^{2}\ps{2}\ldots a^{n}\ps{2} \ot \ldots \\
&\ldots \ot c^{i-1}\cdot a^{i}\ps{i}a^{i+1}\ps{i}\ldots  a^{n}\ps{i} \ot \ve(c^{i}\cdot a^{i+1}\ps{i+1}\ldots  a^{n}\ps{i+1}) \ot c^{i+1}\cdot a^{i+2}\ps{i+2}\ldots  a^{n}\ps{i+2} \ot\ldots \\
& \ldots \ot c^{n-1}\cdot a^n\ps{n}\ot c^n\Big) \ot_{\C{H}} [1\ot a^0a^{1}\ps{0}\ldots a^{n}\ps{0}]  \\
&=  d_i(\Phi((a^0 \ot c^0) \ot_A \ldots \ot_A (a^n \ot c^n) \ot_{A^{\rm e}} \ot 1)).
\end{align*}
 On the next round, we consider the degeneracy operators. For $0 \leq j \leq n$, we have
\begin{align*}
& \Phi(\s_j((a^0 \ot c^0) \ot_A \ldots \ot_A (a^n \ot c^n) \ot_{A^{\rm e}} \ot 1))  \\
&=  \Phi((a^0 \ot c^0) \ot_A \ldots \ot_A \D(a^j \ot c^j) \ot_A \ldots \ot_A (a^n \ot c^n) \ot_{A^{\rm e}} \ot 1)  \\
&=  \Big(c^0\cdot a^{1}\ps{1}\ldots a^{n}\ps{1} \ot  c^1\cdot a^{2}\ps{2}\ldots a^{n}\ps{2} \ot \ldots \\
&\ldots \ot c^{j-1}\cdot a^{j}\ps{j}a^{j+1}\ps{j}\ldots  a^{n}\ps{j} \ot c^{j}\ps{1}\cdot a^{j+1}\ps{j+1}\ldots  a^{n}\ps{j+1} \ot c^{j}\ps{2}\cdot a^{j+2}\ps{j+2}\ldots  a^{n}\ps{j+2} \ot\ldots \\
& \ldots \ot c^{n-1}\cdot a^n\ps{n+1}\ot c^n\Big) \ot_{\C{H}} [1\ot a^0a^{1}\ps{0}\ldots a^{n}\ps{0}]  \\
&=  \Big(c^0\cdot a^{1}\ps{1}\ldots a^{n}\ps{1} \ot  c^1\cdot a^{2}\ps{2}\ldots a^{n}\ps{2} \ot \ldots \\
&\ldots \ot c^{j-1}\cdot a^{j}\ps{j}a^{j+1}\ps{j}\ldots  a^{n}\ps{j} \ot \D(c^{j}\cdot a^{j+1}\ps{j+1}\ldots  a^{n}\ps{j+1}) \ot c^{j+1}\cdot a^{j+2}\ps{j+2}\ldots  a^{n}\ps{j+2} \ot\ldots \\
& \ldots \ot c^{n-1}\cdot a^n\ps{n}\ot c^n\Big) \ot_{\C{H}} [1\ot a^0a^{1}\ps{0}\ldots a^{n}\ps{0}]  \\
&=  s_j(\Phi((a^0 \ot c^0) \ot_A \ldots \ot_A (a^n \ot c^n) \ot_{A^{\rm e}} \ot 1)).
\end{align*}

As for the cyclic operator, we have
\begin{align*}
& \Phi(\tau_n((a^0 \ot c^0) \ot_A \ldots \ot_A (a^n \ot c^n) \ot_{A^{\rm e}} \ot 1))  \\
&=   \Phi((a^n \ot c^n) \ot_A (a^0 \ot c^0) \ot_A \ldots \ot_A (a^{n-1} \ot c^{n-1}) \ot_{A^{\rm e}} \ot 1)  \\
&=  \Big(c^n\cdot a^{0}\ps{1}\ldots a^{n-1}\ps{1} \ot  c^0\cdot a^{1}\ps{2}\ldots a^{n-1}\ps{2} \ot \ldots \\
&\ldots \ot c^{n-2}\cdot a^{n-1}\ps{n}\ot c^{n-1}\Big) \ot_{\C{H}} [1\ot a^na^{0}\ps{0}\ldots a^{n-1}\ps{0}]  \\
&=  \Big(c^n\cdot a^{0}\ps{1}\ldots a^{n-1}\ps{1} \ot  c^0\cdot a^{1}\ps{2}\ldots a^{n-1}\ps{2} \ot \ldots \\
&\ldots \ot c^{n-2}\cdot a^{n-1}\ps{n}\ot c^{n-1}\Big) \ot_{\C{H}} [a^n\ps{1}\ot a^{0}\ps{0}\ldots a^{n-1}\ps{0}a^n\ps{0}]  \\
&=  \Big(c^n\cdot a^{0}\ps{1}\ldots a^{n-1}\ps{1}a^n\ps{1} \ot  c^0\cdot a^{1}\ps{2}\ldots a^{n}\ps{2} \ot \ldots \\
& \ldots \ot c^{n-2}\cdot a^{n-1}\ps{n}a^n\ps{n-1}\ot c^{n-1}\cdot a^n\ps{n}\Big) \ot_{\C{H}} [1\ot a^{0}\ps{0}\ldots a^{n-1}\ps{0}a^n\ps{0}]  \\
&=  t_n(\Phi((a^0 \ot c^0) \ot_A \ldots \ot_A (a^n \ot c^n) \ot_{A^{\rm e}} \ot 1)),
\end{align*}
using Lemma \ref{A-right-to-left-balanced} one last time at the third equality.
\end{proof}

\begin{proposition}\label{thm-A-C-to-C-H-M}
For any Hopf algebra $\C{H}$ with invertible antipode and a right $\C{H}$-module coalgebra $C$ coaction on a right $\C{H}$-comodule algebra~$A$ the map (\ref{A-C-to-C-H-M}) is
an isomorphism of cyclic objects.
\end{proposition}
\begin{proof}
For constructing the inverse, let us define an auxiliary mapping
\begin{align}\label{map-psi}
\widetilde{\Psi}: C^{\ot \bullet+1}\ot \C{H}\ot A &\lra {\rm C}^{A^{\rm e}}_{\bullet}(A \ot C, A) \nonumber\\
 \left(c^0 \odots c^n\right) \ot \left(\mb{h} \ot a\right) &\mapsto \\
 \Big(\hspace{-0.3em}\left(1 \ot c^0 \mb{h}\ps{1}S^{-1}(a\ps{n+1})\right)\hspace{-0.1em} \ot_A \hspace{-0.15em}\left(1 \ot c^1 h\ps{2}S^{-1}(a\ps{n})\right)\hspace{-0.1em} &\ot_A \cdots \ot_A \hspace{-0.15em}\left(1 \ot c^n \mb{h}\ps{n+1}S^{-1}(a\ps{1})\right)\hspace{-0.3em}\Big)\hspace{-0.1em} \ot_{A^{\rm e}}\hspace{-0.1em} a\ps{0}.\nonumber
\end{align}
We first note that for all $a,a' \in A$ and $\mb{h}\in \C{H}$
\begin{align*}
& \widetilde{\Psi}\big(\left(c^0 \odots c^n\right) \ot \left(\mb{h}a'\ps{1}\ot aa'\ps{0}\right)\big)  \\
&=  \Big((1 \ot c^0 \mb{h}\ps{1}a'\ps{1} S^{-1}(a\ps{n+1}a'\ps{0}\ps{n+1})) \ot_A (1 \ot c^1 \mb{h}\ps{2}a'\ps{2} S^{-1}(a\ps{n}a'\ps{0}\ps{n})) \ot_A \ldots \\
& \ldots \ot_A  (1 \ot c^n \mb{h}\ps{n+1}a'\ps{n+1}S^{-1}(a\ps{1}a'\ps{0}\ps{1}))\Big) \ot_{A^{\rm e}} a\ps{0}a'\ps{0}\ps{0}  \\
&= \Big((1 \ot c^0 \mb{h}\ps{1}S^{-1}(a\ps{n+1})) \ot_A (1 \ot c^1 \mb{h}\ps{2}S^{-1}(a\ps{n})) \ot_A \ldots \\
& \ldots \ot_A  (1 \ot c^n \mb{h}\ps{n+1}S^{-1}(a\ps{1}))\Big) \ot_{A^{\rm e}} a\ps{0}a'  \\
&= \Big((1 \ot c^0 \mb{h}\ps{1}S^{-1}(a\ps{n+1})) \ot_A (1 \ot c^1 \mb{h}\ps{2}S^{-1}(a\ps{n})) \ot_A \ldots \\
& \ldots \ot_A  (1 \ot c^n \mb{h}\ps{n+1}S^{-1}(a\ps{1}))\Big) \ot_{A^{\rm e}} (a\ps{0},1)  a'  \\
&= \Big((1 \ot c^0 \mb{h}\ps{1}S^{-1}(a\ps{n+1})) \ot_A (1 \ot c^1 \mb{h}\ps{2}S^{-1}(a\ps{n})) \ot_A \ldots \\
& \ldots \ot_A  (1 \ot c^n \mb{h}\ps{n+1}S^{-1}(a\ps{1}))\Big)  a\ps{0} \ot_{A^{\rm e}} a'  \\
&= \Big((a \ot c^0 \mb{h}\ps{1}) \ot_A (1 \ot c^1 \mb{h}\ps{2}) \ot_A \ldots \ot_A  (1 \ot c^n h\ps{n+1})\Big) \ot_{A^{\rm e}} a'  \\
&= \Big((a \ot c^0 \mb{h}\ps{1}) \ot_A (1 \ot c^1 \mb{h}\ps{2}) \ot_A \ldots \ot_A  (1 \ot c^n \mb{h}\ps{n+1})\Big) \ot_{A^{\rm e}} (a',1) 1  \\
&= (a',1) \Big((a \ot c^0 \mb{h}\ps{1}) \ot_A (1 \ot c^1 \mb{h}\ps{2}) \ot_A \ldots \ot_A  (1 \ot c^n \mb{h}\ps{n+1})\Big) \ot_{A^{\rm e}} a'  \\
&= \Big((a'a \ot c^0 \mb{h}\ps{1}) \ot_A (1 \ot c^1 \mb{h}\ps{2}) \ot_A \ldots \ot_A  (1 \ot c^n \mb{h}\ps{n+1})\Big) \ot_{A^{\rm e}} 1  \\
&= (a'a,1) \Big((1 \ot c^0 \mb{h}\ps{1}) \ot_A (1 \ot c^1 \mb{h}\ps{2}) \ot_A \ldots \ot_A  (1 \ot c^n \mb{h}\ps{n+1})\Big) \ot_{A^{\rm e}} 1  \\
&= \Big((1 \ot c^0 \mb{h}\ps{1}) \ot_A (1 \ot c^1 \mb{h}\ps{2}) \ot_A \ldots \ot_A  (1 \ot c^n\cdot \mb{h}\ps{n+1})\Big) \ot_{A^{\rm e}} a'a  \\
&=  \Big((1 \ot c^0 \mb{h}\ps{1}S^{-1}(a'\ps{n+1}a\ps{n+1})) \ot_A (1 \ot c^1 \mb{h}\ps{2}S^{-1}(a'\ps{n}a\ps{n})) \ot_A \ldots \\
& \ldots \ot_A  (1 \ot c^n \mb{h}\ps{n+1}S^{-1}(a'\ps{1}a\ps{1}))\Big) \ot_{A^{\rm e}} a'\ps{0}a\ps{0}  \\
&= \widetilde{\Psi}\big(\left(c^0 \odots c^n\right) \ot \left(\mb{h}\ot a'a\right)\big),
\end{align*}
where we used Lemma \ref{A-right-to-left-balanced} once more at the eleventh equality, what proves that \eqref{map-psi} factors through the tensor product with the quotient $M$ of $\C{H}\ot A$. Moreover, it is evident that
\[
\widetilde{\Psi}\left((c^0 \odots c^n) \ot (\mb{h}\ot a)\right) = \widetilde{\Psi}\left((c^0 \mb{h}\ps{1} \odots c^n \mb{h}\ps{n+1}) \ot (1\ot a)\right),
\]
what proves that \eqref{map-psi} factors through the tensor product with $M$ balanced over $\C{H}$.
As a result, \eqref{map-psi} induces a well defined map
\begin{align}\label{map-psi-II}
\begin{split}
&\Psi:{\rm C}^{\mathcal{H}}_{\bullet}(C, M) \lra {\rm C}^{A^{\rm e}}_{\bullet}(A \ot C, A)\\
& (c^0 \odots c^n) \ot_{\C{H}} [1 \ot a] \mapsto \\
& \Big((1 \ot c^0 S^{-1}(a\ps{n+1})) \ot_A (1 \ot c^1 S^{-1}(a\ps{n})) \ot_A \ldots \ot_A (1 \ot c^n S^{-1}(a\ps{1}))\Big) \ot_{A^{\rm e}} a\ps{0}.
\end{split}
\end{align}
Finally we see that
\begin{align*}
& \Phi\circ \Psi\Big(\hspace{-0.27em}\left(c^0 \odots c^n\right) \ot_\C{H} [1\ot a]\Big)  \\
&=  \Phi \Big(\hspace{-0.27em}\left(1 \ot c^0 S^{-1}(a\ps{n+1})\right) \ot_A \left(1 \ot c^1 S^{-1}(a\ps{n})\right) \ot_A \ldots \ot_A \left(1 \ot c^n S^{-1}(a\ps{1})\right) \ot_{A^{\rm e}} a\ps{0}\Big) \\
&=  \Phi \Big(\hspace{-0.27em}\left((a \ot c^0) \ot_A (1 \ot c^1) \ot_A \ldots \ot_A (1 \ot c^n)\right) \ot_{A^{\rm e}} 1\Big)  \\
&=  \left(c^0 \odots c^n\right) \ot_\C{H} [1\ot a],
\end{align*}
and that
\begin{align*}
& \Psi\circ \Phi\Big(\hspace{-0.27em}\left((1 \ot c^0) \ot_A (1 \ot c^1) \ot_A \cdots \ot_A (1  \ot c^n)\right) \ot_{A^{\rm e}} a\Big)  \\
&=  \Psi \Big(\hspace{-0.27em}\left(c^0 \odots  c^n\right) \ot_\C{H} [1\ot a]\Big)  \\
&=  \left((a \ot c^0) \ot_A (1 \ot c^1) \ot_A \cdots \ot_A (1  \ot c^n)\right) \ot_{A^{\rm e}} 1  \\
&=  a \left((1 \ot c^0) \ot_A (1 \ot c^1) \ot_A \cdots \ot_A (1  \ot c^n)\right) \ot_{A^{\rm e}} 1  \\
&= \left((1 \ot c^0) \ot_A (1 \ot c^1) \ot_A \cdots \ot_A (1  \ot c^n)\right) \ot_{A^{\rm e}} a
\end{align*}
which proves that \eqref{A-C-to-C-H-M} and \eqref{map-psi-II} are inverse to each other.
\end{proof}

\begin{corollary}
Composing \eqref{A-C-to-C-H-M} with 
\eqref{cyclic-mod-A-tensor-C} and \eqref{alg-ext-Z-to-aux-iso-coring-FY}, we obtain a \emph{characteristic map}
\begin{align}\label{char}
\begin{gathered}
{\rm C}_{\bullet}(A\hspace{0.2em}|\hspace{0.1em}B)\rightarrow  {\rm C}^{\mathcal{H}}_{\bullet}(C, M),\\
a^0 \widehat{\ot}_B a^1 \widehat{\ot}_B \cdots \widehat{\ot}_B  a^n \mapsto \Big(e a^{1}\ps{1}\ldots a^{n}\ps{1} \ot  e a^{2}\ps{2}\ldots a^{n}\ps{2} \ot\ldots\\
\ \ \ \ \ \ldots \ot e a^{n-1}\ps{n-1}a^n\ps{n-1}  \ot e a^n\ps{n}\ot e\Big) \ot_{\C{H}} [1\ot a^0a^{1}\ps{0}\ldots a^{n}\ps{0}]
\end{gathered}
\end{align}
which is a map of cyclic objects. If a Hopf algebra 
$\mathcal{H}$ has invertible antipode,  a right 
$\mathcal{H}$-module coalgebra $C$ and a right 
$\mathcal{H}$-comodule algebra $A$ are 
$\mathcal{H}$-entwined and satisfy the Galois condition, the map (\ref{char}) of cyclic objects is an isomorphism.
\end{corollary}
The next proposition rewrites the above definition of the characteristic map as commutativity of the following square.
\begin{proposition}
The  diagram of  
cyclic vector spaces 
\begin{align}\label{q-iso-cyc-fun}
\begin{array}{ccc}
 {\rm C}^{\mathcal{H}}_{\bullet}(C, M) &\longleftarrow &{\rm C}^{A^{\rm e}}_{\bullet}(A \ot C, A)\\
\uparrow&   &\uparrow\\
{\rm C}_{\bullet}(A\hspace{0.2em}|\hspace{0.1em}B) &\longleftarrow & {\rm C}^{A^{\rm e}}_{\bullet}(A \ot_B\hspace{-0.175em} A, A)
\end{array}
\end{align}
with the left vertical arrow being the characteristic map \eqref{char}, the right vertical arrow being induced by the canonical map \eqref{entw-can} of corings, 
and  
the top and bottom arrows are \eqref{A-C-to-C-H-M}
and \eqref{aux-iso-coring-FY-to-alg-ext-Z}, respectively, commutes.
\end{proposition}
\begin{proof}
Since the map \eqref{alg-ext-Z-to-aux-iso-coring-FY} is inverse to \eqref{aux-iso-coring-FY-to-alg-ext-Z}, the definition of the characteristic map as the composition of\eqref{A-C-to-C-H-M} with 
\eqref{cyclic-mod-A-tensor-C} and \eqref{alg-ext-Z-to-aux-iso-coring-FY} reads as commutativity of \eqref{q-iso-cyc-fun}.
\end{proof}
\begin{proposition}
If the antipode of the Hopf algebra $\mathcal{H}$ is bijective, the commutative square \eqref{q-iso-cyc-fun} is an isomorphism between the map \eqref{cyclic-mod-A-tensor-C}  induced by the canonical map \eqref{entw-can} of corings, and the characteristic map \eqref{char}.
\end{proposition}
\begin{proof}
The botom arrow \eqref{alg-ext-Z-to-aux-iso-coring-FY} in the commutative square \eqref{q-iso-cyc-fun} is always an isomorphism. If the antipode of the Hopf algebra $\mathcal{H}$ is bijective, the top arrow is an isomorphism as well. Since the horizontal arrows are isomorphisms, the square is an isomorphism between the vertical arrows.
\end{proof}
\subsection{The Galois case} The Hopf-Galois case is very special for two different reasons. First of all, the inertia is trivial, second, the coalgebra-encoded Galois symmetry coincides with the Hopf-algebra-encoded entwining symmetry. Classically speaking, it restricts the situation to the context of a free group action. However, showing that our characteristic map in this apparently trivial case reduces to the well known  
Jara-\c{S}tefan isomorphism of cyclic objects, requires some effort. Note the moment in the proof when we have to assume that the antipode of the Hopf algebra is invertible.

We start this subsection by proving  that the SAYD module of Theorem \ref{thm-SAYD} generalizes the one introduced by Jara and \c{S}tefan in \cite{JaraStef06} for Hopf-Galois extensions. Then in our isomorphism of cyclic objects we can restrict to the case of the right $\C{H}$-module coalgebra $C$ equal to the Galois Hopf algebra being $\C{H}$ itself.

For that purpose we need the following simple lemmas, playing also a role in defining the Miyashita-Ulbrich action. For the sake of  fixing the conventions, and avoiding possible confusion of left and right, especially after applying the flip of tensorands, we shall prove the lemmas in detail.

\begin{lemma}
For any right Hopf-Galois extension $A$ of an algebra $B$, with the Galois Hopf algebra $\C{H}$ and the translation map
\[
\tau: \C{H}\to A\ot_{B}A,\qquad  \mb{h}\mapsto \mb{h}^{[1]}\ot_{B} \mb{h}^{[2]}
\]
one has
\begin{align}\label{tau hh'}
(\mb{h}\mb{h}')^{[1]}\ot_{B} (\mb{h}\mb{h}')^{[2]}=\mb{h}'^{[1]}\mb{h}^{[1]}\ot_{B} \mb{h}^{[2]}\mb{h}'^{[2]}.
\end{align}
\end{lemma}

\begin{proof}
By the definition of the translation map as the restriction of the inverse $can^{-1}$ to the canonical map $can$ to $1\otimes \C{H}\subseteq A\ot \C{H}$  
one has
\begin{align}\label{trans-can}
1\ot \mb{h}=can(can^{-1}(1\ot \mb{h}))=can(\mb{h}^{[1]}\ot_{B}\mb{h}^{[2]})=\mb{h}^{[1]}\mb{h}^{[2]}\ps{0}\ot \mb{h}^{[2]}\ps{1}.
\end{align}
Applying \eqref{trans-can}  to $1\ot \mb{h}\mb{h}'$ one gets
\begin{align}\label{LHS hh'}
1\ot \mb{h}\mb{h}'=(\mb{h}\mb{h}')^{[1]}(\mb{h}\mb{h}')^{[2]}\ps{0}\ot (\mb{h}\mb{h}')^{[2]}\ps{1}.
\end{align}
Applying \eqref{trans-can} to $1\ot \mb{h}'$ in  $1\ot \mb{h}\mb{h}'=(1\ot \mb{h})(1\ot \mb{h}')$, inscribing $1$ between the two left-most tensorands,  applying \eqref{trans-can} to $1\ot \mb{h}$ and finally using the definition of a comodule algebra one gets
\begin{align}\label{RHS hh'}
\begin{split}
&1\ot \mb{h}\mb{h}'=\mb{h}'^{[1]}\mb{h}'^{[2]}\ps{0}\ot \mb{h}\mb{h}'^{[2]}\ps{1}=\mb{h}'^{[1]}1\mb{h}'^{[2]}\ps{0}\ot \mb{h}\mb{h}'^{[2]}\ps{1}\\
&=\mb{h}'^{[1]}\mb{h}^{[1]}\mb{h}^{[2]}\ps{0}\mb{h}'^{[2]}\ps{0}\ot \mb{h}^{[2]}\ps{1}\mb{h}'^{[2]}\ps{1}=\mb{h}'^{[1]}\mb{h}^{[1]}(\mb{h}^{[2]}\mb{h}'^{[2]})\ps{0}\ot (h^{[2]}\mb{h}'^{[2]})\ps{1}.
\end{split}
\end{align}
Now, applying  $can^{-1}$ to the equality
\[
(\mb{h}\mb{h}')^{[1]}(\mb{h}\mb{h}')^{[2]}\ps{0}\ot (\mb{h}\mb{h}')^{[2]}\ps{1}=\mb{h}'^{[1]}\mb{h}^{[1]}(\mb{h}^{[2]}\mb{h}'^{[2]})\ps{0}\ot (\mb{h}^{[2]}\mb{h}'^{[2]})\ps{1},
\]
resulting from comparing the right hand sides of \eqref{LHS hh'} and  \eqref{RHS hh'}, one obtains the desired identity.  
\end{proof}

\begin{lemma} 
The image of the translation map belongs to the $B$-centralizer $(A\otimes_{B}A)^{B}$ of the $A$-bimodule $A\otimes_{B}A$, namely,
\begin{align}\label{B centr A ot A}
\tau (\C{H})\subseteq (A\otimes_{B}A)^{B}.
\end{align}
\end{lemma}

\begin{proof}
Let us multiply the element $\tau(\mb{h})=\mb{h}^{[1]}\ot \mb{h}^{[2]}$ from both sides by $b\in B=A^{co\ H}\subseteq A$. 
\begin{align*}
b\tau(\mb{h})&=b\mb{h}^{[1]}\ot \mb{h}^{[2]}=bcan^{-1}(1\ot \mb{h})=can^{-1}(b\ot \mb{h}),\\
\tau(\mb{h})b&=\mb{h}^{[1]}\ot \mb{h}^{[2]}b=can^{-1}(1\ot \mb{h})b=can^{-1}((1\ot \mb{h})b)\\
&=can^{-1}(b\ps{0}\ot \mb{h}b\ps{1})=can^{-1}(b\ot \mb{h}),
\end{align*}
hence $b\tau(\mb{h})=\tau(\mb{h})b$.
\end{proof}

We shall also need the following lemma.

\begin{lemma}
Given a Hopf-Galois extension $B \subseteq A$, and $a\in A$, the identity
\begin{align}\label{a ot 1}
a\ot 1=a\ps{1}^{\hspace{0.5em}[2]}\ot a\ps{0}a\ps{1}^{\hspace{0.5em}[1]}
\end{align}
holds in $(A\ot A)_{B}$.
\end{lemma}

\begin{proof} 
It suffices to prove the flipped version in $A\ot_{B}A$, that is,
\[
1\ot_{B}a=can^{-1}(can(1\ot_{B}a))=can^{-1}(a\ps{0}\ot a\ps{1})= a\ps{0}a\ps{1}^{\hspace{0.5em}[1]}\ot_{B}a\ps{1}^{\hspace{0.5em}[2]}.
\]
\end{proof}
In the following proposition, for any $B$-bimodule $N$ we use the symbol $N_{B}$ to denote the quotient space $N/[B,N]$.  
\begin{proposition}
For any right Hopf-Galois extension $A$ of an algebra $B$, with the Galois Hopf algebra $\C{H}$ with invertible antipode, the SAYD module $M$ is isomorphic to the quotient space $A_{B}$ with its canonical structure of an SAYD module.
\end{proposition}
\begin{proof}
Applying the bijective linear map (the flipped $can^{-1}$) 
\begin{align}\label{alpha iso}
\gamma: \C{H}\ot A\rightarrow (A\ot A)_{B},\ \ \ \ \mb{h}\ot a\mapsto \mb{h}^{[2]}\ot a\mb{h}^{[1]}
\end{align}
to $\mb{h}a'\ps{1}\ot aa'\ps{0} \in H \ot A$ and next the flipped version of the identity \eqref{tau hh'}, we obtain 
\begin{align*}
&\gamma (\mb{h}a'\ps{1}\ot aa'\ps{0})=(\mb{h}a'\ps{1})^{[2]}\ot aa'\ps{0}(\mb{h}a'\ps{1})^{[1]}\\
& = \mb{h}^{[2]}a'\ps{1}^{\hspace{0.5em}[2]}\ot aa'\ps{0}a'\ps{1}^{\hspace{0.5em}[1]}\mb{h}^{[1]}=\mb{h}^{[2]}a'\ot a\mb{h}^{[1]}.
\end{align*}
in view of \eqref{a ot 1}. Note that the application of \eqref{a ot 1} to prove the latter equality 
is legitimate by the flipped version of \eqref{B centr A ot A}.

We have also
\[
\gamma(\mb{h}\ot a'a)=\mb{h}^{[2]}\ot a'a\mb{h}^{[1]},
\]
hence finally 
\[
\gamma(\mb{h}a'\ps{1}\ot aa'\ps{0}-\mb{h}\ot a'a)=\mb{h}^{[2]}a'\ot a\mb{h}^{[1]}-\mb{h}^{[2]}\ot a'a\mb{h}^{[1]}.
\]
Since by \eqref{alpha iso} the elements $\mb{h}^{[2]}\ot a\mb{h}^{[1]}$ span $(A\ot A)_{B}$, the latter means that \eqref{alpha iso} induces the isomorphism
\begin{align}\label{SAYD iso} 
M\to (A\ot_{A}\hspace{-0.15em}A)_{B}=A_{B},\ \ \ \ \mb{h}\ot a\mapsto \mb{h}^{[2]}a\mb{h}^{[1]}.
\end{align}
Note that by the flipped version of \eqref{tau hh'}, the latter isomorphism is left $\C{H}$-linear. Its right $\C{H}$-colinearity, on the other hand, follows from the fact that the translation map is a morphism of $\C{H}$-bicomodules \cite[Prop. 3.6]{Brze96-II} which is equivalent to the identity 
\[
\mb{h}\ps{1}\ot \mb{h}\ps{2}^{\hspace{0.5em}[1]} \ot_{B}\mb{h}\ps{2}^{\hspace{0.5em}[2]}\ot  \mb{h}\ps{3} 
= S^{-1}(\mb{h}^{[1]}\ps{1})\ot \mb{h}^{[1]}\ps{0}\ot_{B} \mb{h}^{[2]}\ps{0}\ot \mb{h}^{[2]}\ps{1}
\]
in $\C{H}\ot A\ot_{B}A \ot \C{H}$. Applying the antipode $S:\C{H}\to \C{H}$ to the left most tensorand, next  moving the first  tensorand to the last  position, and finally flipping the first two  tensorands of both sides of the resulting identity, we arrive at the identity
\[
\mb{h}\ps{2}^{\hspace{0.5em}[2]}\ot \mb{h}\ps{2}^{\hspace{0.5em}[1]}\ot  \mb{h}\ps{3} \ot S(\mb{h}\ps{1})   
=  \mb{h}^{[2]}\ps{0}\ot \mb{h}^{[1]}\ps{0}\ot \mb{h}^{[2]}\ps{1}\ot \mb{h}^{[1]}\ps{1}
\]
in $ (A\ot A)_{B}\ot \C{H} \ot \C{H}$. The latter implies 
\[
\mb{h}\ps{2}^{\hspace{0.5em}[2]}a\ps{0} \mb{h}\ps{2}^{\hspace{0.5em}[1]}\ot  \mb{h}\ps{3} a\ps{1} S(\mb{h}\ps{1})   
=  \mb{h}^{[2]}\ps{0}a\ps{0} \mb{h}^{[1]}\ps{0}\ot \mb{h}^{[2]}\ps{1}a\ps{1} \mb{h}^{[1]}\ps{1}
\]
in $ A_{B}\ot \C{H} $, which amounts to the fact that the map \eqref{SAYD iso} of SAYD modules is also $\C{H}$-colinear. 
\end{proof}
\subsection{Module-coalgebra-Galois extensions: quantized  Hopf principal bundles}
In the case of quantizations of principal bundles  beyond the Hopf-Galois context, our characteristic map is the extension of the Jara-\c{S}tefan isomorphism \cite{JaraStef06} to the context of module-coalgebra-Galois extensions. The most important case requiring such generality is the quantum instanton bundle, or a quantum  Hopf principal bundle, introduced in \cite{BoneCiccTarl02} as a quantization of the classical  counterpart and proven to be a module coalgebra-Galois extension in \cite{BoneCiccDabrTarl04}\footnote{It was note explicitly stated there but follows easily from the formulas therein.}, and the quantum Hopf principal bundle over the Podle\'{s} sphere \cite{Pod87, Brze96, Brze97}. 



\begin{theorem}\label{thm-HP} For any pair $(C, A)$ of an $\C{H}$-entwined right coaugmented $\C{H}$-module coalgebra $C$ and a $C$-Galois right $\C{H}$-comodule algebra  extension $A$  of an algebra $B$, the relative periodic cyclic homology of the extension $B\subseteq A$ is isomorphic to the Hopf-cyclic homology of the right $\C{H}$-module coalgebra $C$ with SAYD-coefficients in $M$. In short,
\[
HP_{\bullet}(A\hspace{0.2em}|\hspace{0.1em}B)\cong HP^{\C{H}}_{\bullet}(C, M). 
\]
\end{theorem}
Finally, we note that it would require further investigation to study what happens to the Gauss-Manin connection under this isomorphism.
\section{Applications}\label{Special cases}
In the present section we explain that the SAYD module constructed in Theorem~\ref{thm-SAYD} is a noncommutative counterpart of the Brylinski $G$-scheme. Therefore, it provides an SAYD-quantization of the inertia groupoid regarded as the action groupoid of the Brylinski $G$-scheme.   Next, we specialise our construction to the case of the universal Hopf algebra  coaction of a finite dimensional algebra. Then the resulting Hopf-cyclic homology provides an invariant of a finite dimensional algebra. Finally, we show that our construction, restricted to the case of  a Hopf-Galois extensions, recovers the SAYD module constructed by 
Jara and \c{S}tefan in \cite{JaraStef06}. The latter implies that in the Hopf-Galois case our characteristic map (\ref{char}) restricts to the  Jara--\c{S}tefan isomorphism of cyclic objects~\cite{JaraStef06}.

\subsection{Inertial Hopf-cyclic quantization of the cyclic nerve via cyclic duality.} 
\subsubsection{SAYD-quantization of the inertia groupoid.} 
 If the algebra $A$ and  the Hopf algebra $\C{H}$ are commutative, the affine scheme  $X:=Spec(A)$  is acted on by the affine group scheme  $G:=Spec(\C{H})$. The inertia groupoid of that action is an action groupoid of the Brylinski $G$-scheme. Then, the following theorem identifies our SAYD module with the algebra of regular functions on the Brylinski scheme.

\begin{proposition}\label{sayd-bryl}
If the Hopf algebra $\C{H}$ and the right $\C{H}$-comodule algebra $A$ are commutative, the SAYD module $M$ becomes a commutative right $\C{H}$-comodule algebra.  If $G=Spec(\C{H})$ is a corresponding affine group scheme and  $X=Spec(A)$ a corresponding affine $G$-scheme, $Spec(M)$ is the Brylinski $G$-scheme, on points defined as
\[
Bry_{G}(X):=\{ (g, x)\in Ad(G)\times X\ \mid\ xg=x\},
\]
with the diagonal right $G$-action by right conjugations on $Ad(G)=G$ and a given right $G$-action on $X$, equipped with a $G$-equivariant morphism into $Ad(G)$.  
\end{proposition}
\begin{proof}
Looking at the definition \eqref{M} of $M$ one sees that in the coproduct  $\C{H}\ot A$ of commutative algebras we have a decomposition 
\[ \mb{h}a'\ps{1}\ot aa'\ps{0}-\mb{h}\ot a'a = (\mb{h}\ot a)(a'\ps{1}\ot a'\ps{0}-1\ot a').\]
Therefore, the algebra structure of the quotient $M$ comes from the fact that in the algebra $\C{H}\otimes A$ the $\C{H}$-submodule of relations defining $M$ is an ideal in $\C{H}\ot A$
generated by elements of the form
\[
a'\ps{1}\ot a'\ps{0}-1\ot a'.
\]
It defines a closed $G$-invariant subscheme of the $G$-scheme $G\times X$ with the diagonal $G$-action by conjugations on $G$, which results from evaluation the $H$-comodule structure on points. By evaluating on points the $H$-module structure gives a map  into $Ad(G)$ induced by the restriction of the projection  $Ad(G)\times X\rightarrow  Ad(G)$. By the AYD-property the map into $Ad(G)$ is $G$-equivariant. The stability  property of the SAYD module is equivalent to the fact that in the pair $(g, x)$ the group element $g$ stabilizes $x$.  
\end{proof}
\subsubsection{Cyclic-dual Hopf-cyclic quantization of the functor induced map of cyclic nerves} Assume now that $C:=\mathscr{O}(H)$ is the underlying right $\mathcal{H}$-module coalgebra of a monoid scheme $H$ equipped with an affine monoid scheme morphism $H\rightarrow G$. We will call this context \emph{the classical case}. Note that in this case the antipode of $\mathcal{H}$ is obviously invertible. Then the following theorem relates our characteristic map (\ref{char}) with the functor induced map of cyclic nerves and cyclic duality in the commutative case.
\begin{proposition}
In the classical case, after applying the composition of the functor $\mathscr{O}(-)$ of global regular functions and the cyclic duality to the commutative diagram (\ref{iso-cyc-fun}) of affine schemes over a ground field, we obtain the commutative diagram \eqref{q-iso-cyc-fun} of  
cyclic vector spaces. 
\end{proposition} 
\begin{proof} The contravariant functor $\mathscr{O}(-)$ transforms the group action on a space into a Hopf algebra module algebra structure and the constant map and the diagonal embedding of an affine scheme into the unit and the multiplication of the algebra of regular functions. Therefore, by the definition of the cyclic objects and morphisms in the diagram  (\ref{iso-cyc-fun}) and Proposition \ref{sayd-bryl}, $\mathscr{O}(-)$ transforms it into a diagram of cocyclic objects. Looking at its cofaces, codegeneracies and cocyclic map one sees that the resulting diagram is cyclic dual to \eqref{q-iso-cyc-fun}.  
\end{proof}
\noindent{\bf Remark.} It means that our characteristic map \eqref{char} can be regarded as a \emph{Hopf-cyclic quantization} of the cocyclic dual of the internal functor induced map (\ref{iso-cyc-fun}) of cyclic nerve schemes.

\subsection{Geometry of the quotient map versus inertia of the action}
All schemes and their morphisms in this section are defined over the field $\mathbb{C}$ of complex numbers. Let there be given a right action of an affine group scheme $G$ on an affine scheme $X$, admitting   a quotient map $f: X\rightarrow X/G=:Y$. Then for $C=\mathcal{H}:=\mathscr{O}(G)$, $A:=\mathscr{O}(X)$, $B:=\mathscr{O}(X/G)=A^{{\rm co}\hspace{0.1em}C}$
The $B$-relative periodic cyclic homology ${\rm HP}_\bullet (A\hspace{0.2em}|\hspace{0.1em}B)$ is an invariant of the geometry of the quotient map $f$. To study $HP_{\bullet}(A\hspace{0.2em}|\hspace{0.1em}B)$ one can resort to the result of Katz that existence of an (integrable) connection on a coherent sheaf implies its local freeness~\cite{Katz70}. In the case of $HP_{\bullet}(A\hspace{0.2em}|\hspace{0.1em}B)$ it is the Gauss--Manin connection \cite{Getz93, Tsyg07, PetrVainVolo18, GinSched12} whose existence, in quite wide generality, is guaranteed by the following result of 
Ginzburg and Schedler.
\begin{theorem}[Thm 4.4.1, \cite{GinSched12}]\label{GinSched}
Let $B$ be a commutative algebra. Let A be an associative algebra equipped with
a central algebra embedding $B\hookrightarrow A$ such that the quotient $A/B$ is a free $B$-module.
Then, there is a canonical flat connection $\nabla_{GM}$ on $HP_{\bullet}(A\hspace{0.2em}|\hspace{0.1em}B)$.
\end{theorem}

 
\subsubsection{The smooth case: free actions and topology} If $f$ is a smooth morphism of smooth varieties, by Andr\'{e}'s $B$-relative Hochschild--Kostant--Rosenberg theorem \cite{And74}, \break
$HP_{\bullet}(A\hspace{0.2em}|\hspace{0.1em}B)$,  is isomorphic to $\mathbb{Z}/2\mathbb{Z}$-graded $B$-relative de Rham cohomology $H_{\rm dR}^{\bullet}(A\hspace{0.2em}|\hspace{0.1em}B)$. The latter admits analytification to become an invariant of the analytic submersion $f^{\rm an}: X^{\rm an}\rightarrow~ Y^{\rm an}$ of complex manifolds. It is an analytic $\mathbb{Z}/2\mathbb{Z}$-graded
vector bundle with fibers being the de~Rham cohomology of the fibers of $f^{\rm an}$. When equipped with the Gauss-Manin integrable connection $\nabla_{GM}$, its subsheaf of $\nabla_{GM}$-parallel sections of its analytification coincides with  the $\mathbb{Z}/2\mathbb{Z}$-graded local system $R^{\bullet}f^{\rm an}_{*}\mathbb{C}_{X^{\rm an}}$ spanning, by integrability of $\nabla_{GM}$, that analytic $\mathbb{Z}/2\mathbb{Z}$-graded vector bundle. Therefore, in this case, the analytification of $HP_{\bullet}(A\hspace{0.2em}|\hspace{0.1em}B)$ depends only on the relative euclidean topology of~$f^{\rm an}$.

In particular, if the $G$-action is free, we are in the commutative context of Hopf-Galois extensions, when the characteristic map is an isomorphism of $B$-modules
\begin{align}\label{char_per_iso}
HP_{\bullet}(A\hspace{0.2em}|\hspace{0.1em}B)\stackrel{\cong}{\longrightarrow} HP^{\C{H}}_{\bullet}(C, M)
\end{align}
\noindent{\bf Remark.} In this case, after analytification of the characteristic map \eqref{char_per_iso}, the invariant of the free action  coincides with the invariant of topology of the smooth quotient map.
\subsubsection{A non-smooth case: inertia and ramification} The following example shows that sometimes, using Theorem \ref{GinSched}, we can study $HP_{\bullet}(A\hspace{0.2em}|\hspace{0.1em}B)$ using $H_{\rm dR}^{\bullet}(A\hspace{0.2em}|\hspace{0.1em}B)$ also in a non-smooth case, e. g. when the quotient map is the following branched covering.

\noindent{\bf Example.} In the case of a double Galois branched covering with $A=B[x]/(x^{2}-b)$ and $b$ not a zero-divisor, acted on by $G=\Bmu_{2}$ \big(with the group of complex points $G(\mathbb{C})=\mathbb{Z}/2\mathbb{Z}$ and the algebra of regular functions $C=\mathcal{H}=\mathscr{O}(G)=\mathbb{C}[t]/(t^{2}-1)$\big) by the reflection $x\mapsto -x$, we have $A^{G}=A^{\rm co\hspace{0.075em}C}=B\subset A = B\oplus Ba$ where $a:=x + (x^{2}-b)\in A$ is not a zero-divisor,  and $B\hookrightarrow A$ defines a branched covering of degree two with the branch locus defined by the ideal $(b)\lhd B$. Since, as the differential 
graded-commutative algebra, the de~Rham complex has the following presentation
\begin{align*}
\Omega^{\bullet}_{A/B}=B[x, dx]/(x^{2}-b, xdx), \end{align*}
$\Omega^{0}_{A/B}=0$ for $p>1$, and the only non-zero de~Rham differential \begin{align}\label{dR}
\begin{array}{ccc}
\Omega^{0}_{A/B} & \stackrel{d}{\longrightarrow} & \Omega^{1}_{A/B}\\
\parallel                 &                  &  \parallel\\           
B\oplus Ba             & \longrightarrow  &  \big(B/(b)\big)da,\\
b_{0}+b_{1}a        & \longmapsto      & \big(b_{1}+(b)\big) da 
\end{array}
\end{align}
 is surjective, hence $H_{\rm dR}^{p}(A\hspace{0.2em}|\hspace{0.1em}B)  =  0$ for $p>0$. From \eqref{dR} we infer also that  
$H_{\rm dR}^{0}(A\hspace{0.2em}|\hspace{0.1em}B)  =  B\oplus (b)a\subset A$ is a free $B$-module. On the other hand, since $A/B\cong Ba$ is free, by Theorem~\ref{GinSched} and the Katz result, $HP_{\bullet}(A\hspace{0.2em}|\hspace{0.1em}B)$ is a locally free $B$-module. Finally, since both $HP_{\bullet}(A\hspace{0.2em}|\hspace{0.1em}B)$ and $H_{\rm dR}^{\bullet}(A\hspace{0.2em}|\hspace{0.1em}B)$ commute with localization of $B$, outside the branch locus, where $f$ is smooth (\'{e}tale), we have the isomorphism of $B[b^{-1}]$-modules
\begin{align*}
& HP_{\bullet}(A\hspace{0.2em}|\hspace{0.1em}B)\ot_{B}\hspace{-0.175em}B[b^{-1}]\cong HP_{\bullet}(A\ot_{B}\hspace{-0.175em}B[b^{-1}]\hspace{0.2em}|\hspace{0.1em}B[b^{-1}])\\
\cong\hspace{0.3em} & H_{\rm dR}^{\bullet}(A\ot_{B}\hspace{-0.175em}B[b^{-1}]\hspace{0.2em}|\hspace{0.1em}B[b^{-1}])\cong H_{\rm dR}^{\bullet}(A\hspace{0.2em}|\hspace{0.1em}B)\ot_{B}\hspace{-0.175em}B[b^{-1}]\cong B[b^{-1}]\oplus B[b^{-1}],
\end{align*}
hence $HP_{\bullet}(A\hspace{0.2em}|\hspace{0.1em}B)$
is a locally free $B$-module of rank two, free outside the branch locus. 


The SAYD module of the algebra of regular functions on the Brylinski scheme is 
\begin{align*}
M=\mathscr{O}\hspace{-0.0175em}\big({\rm Bry}_{G}(X)\big) = B[t,x]/\left(t^{2}-1, x^{2}-b, (t-1)x\right) \cong B\oplus B\delta \oplus Ba \cong B\ot \big(\mathcal{H}\oplus \mathbb{C}_{\varepsilon, \delta}\big)
\end{align*}
where $\delta:= t\hspace{0.2em} {\rm mod}\hspace{0.2em} (t^{2}-1)$ is a grouplike in $\mathcal{H}$, and the utmost right direct summands $\mathcal{H}$ and $\mathbb{C}_{\varepsilon, \delta}$ are  SAYD modules being the Hopf algebra itself, with its standard SAYD module, and the one dimensional SAYD module defined by the modular pair in involution $\left(\varepsilon, \delta\right)$, respectively. Therefore,
\begin{align*}
HP^{\C{H}}_{\bullet}(\C{H}, M) = B\ot\big(HP^{\C{H}}_{\bullet}(\C{H}, \C{H}) \oplus HP^{\C{H}}_{\bullet}(\C{H}, \mathbb{C}_{\varepsilon, \delta}) \big)
\end{align*}
is a free $B$-module,  spanned by the Hopf-cyclic homology dependent only on $\C{H}$ and its grouplike~$\delta$. Since over $B[b^{-1}]$ the characteristic map is an isomorphism of free $B[b^{-1}]$-modules, by tensoring with the residue field $B[b^{-1}]/\mathfrak{m}$ we obtain an isomorphism of vector spaces
\[
HP^{\C{H}}_{\bullet}(\C{H}, \C{H}) \oplus HP^{\C{H}}_{\bullet}(\C{H}, \mathbb{C}_{\varepsilon, \delta})\cong \mathbb{C}^{2}
\]
which provides  the isomorphism
\[
HP^{\C{H}}_{\bullet}(\C{H}, M)\cong B^{2}.
\]
It implies that the characteristic map
\begin{align}\label{char_per}
HP_{\bullet}(A\hspace{0.2em}|\hspace{0.1em}B)\longrightarrow HP^{\C{H}}_{\bullet}(C, M)
\end{align}
is an injective morphism from a locally free $B$-module to a free module, both of rank two, whose cokernel is supported on the branch locus. 

\noindent{\bf Remark.} In this case, after analytification of the characteristic map \eqref{char_per}, the invariant of the inertia modulo the invariant of topology of the maximal free action is an invariant of ramification of the quotient map, supported on the branch locus.
\subsection{The canonical inertial Hopf-cyclic object of a finite dimensional algebra}
In the present subsection we will use the fact that for a finite dimensional algebra $A$ there is a universal right Hopf-algebra coaction on $A$. The  universal Hopf algebra is  Manin's  Hopf envelope \cite{Man88} of  Tambara's coendomorphism bialgebra \cite{Tam90} of the finite dimensional algebra $A$. Universality of the coaction can be viewed as a commutative  diagram as  (\ref{diagram}), where $\widetilde{\alpha}$ is a universal  Hopf algebra $\widetilde{\C{H}}$ coaction  on $A$, which means that for any other Hopf algebra $\C{H}$ coaction $\alpha$ on $A$ there exists a unique Hopf algebra map $\varphi: \widetilde{\C{H}}\rightarrow \C{H}$ making the diagram  (\ref{diagram}) commute. Therefore, our inertial cyclic object for any Hopf algebra coacting on $A$  admits a canonical inertial cyclic object map from such a cyclic object defined for the universal Hopf algebra coaction. We call the latter the \emph{universal inertial cyclic module}. Moreover, the latter cyclic object is an invariant of the finite dimensional algebra. In particular, finite dimensional algebras with different universal inertial cyclic, negative cyclic or periodic cyclic homology or cyclic-dual cohomology, cannot be isomorphic. In particular, given finite-dimensional algebra, our characteristic map maps  
the periodic cyclic homology to our universal inertial periodic Hopf-cyclic homology. We expect that the latter can distinguish algebras with the same periodic cyclic homology, a poor invariant of finite dimensional algebras which, by the Goodwillie theorem \cite{good}, is insensitive to nilpotent extensions.

\vspace{-0.5em}
\section*{Acknowledgements}
\noindent This work was partially supported by NCN grant UMO-2015/19/B/ST1/03098.

\bibliographystyle{plain}
\bibliography{references}{}
\end{document}